\documentclass[A4paper,10pt,draft]{amsart}
\usepackage{amssymb,amsmath,amsfonts,amsthm}
\usepackage[utf8]{inputenc}
\usepackage[english]{babel}
\usepackage[dvipsnames]{xcolor}
\usepackage{mathtools}
\mathtoolsset{showonlyrefs}

\newtheorem{theorem}{Theorem}[section]
\newtheorem{corollary}{Corollary}[section]
\newtheorem{lemma}{Lemma}[section]
\newtheorem{proposition}{Proposition}[section]
\newtheorem{definition}{Definition}[section]
\newtheorem{remark}{Remark}
\newtheorem{example}{Example}

\def\<{\langle}
\def\>{\rangle}
\begin{document}
\title[Gabor Frames]{Pseudodifferential operators on time-frequency invariant Banach spaces and applications to Gabor Frames}
\author{Gianluca Garello, Alessandro  Morando}
\maketitle
\begin{abstract}
Starting from the study of pseudodifferential operators with completely periodic symbols,  we obtain results of continuity and invertibility of a class of Gabor operators on time-frequency invariant Banach spaces. As an applications we find sufficient conditions for the existence of Gabor frames on $L^2$, associated to a general lattice $L \mathbb Z^{2d}$, where $L$ is an invertible square matrix.

\end{abstract}

\maketitle
\tableofcontents
\section{Introduction}\label{INT}
Let us consider a continuous signal
$
f=f(t)\in L^2(\mathbb R^d),
$
(where $t$ stands for \lq\lq time\rq\rq). The continuous {\em Time-Frequency} (TF) representation of $f$ is
\begin{equation}\label{TFrep}
	f=\frac1{(g,\gamma)}\iint_{\mathbb R^d\times\mathbb R^d}V_gf(x,\omega)M_\omega T_x\gamma\, dx\,d\omega\,, \quad g,\gamma\in L^2(\mathbb R^d), \,\, (g,\gamma)\ne 0.
\end{equation}
In \eqref{TFrep}, $(\cdot,\cdot)$ denotes the inner product in $L^2(\mathbb R^d)$, with norm $\Vert f\Vert^2_{2}:=(f,f)$ and $V_g f(x,\omega):=(f, M_\omega T_x g)$ is the  {\em short-time Fourier transform} (STFT) with respect to the {\em analysis window} $g=g(t)\in L^2(\mathbb R^d)$.
$M_\omega T_xh(t):=e^{2\pi i\omega\cdot t} h(t-x)$ is the {\em time-frequency shift} of a function $h=h(t)$, with parameter $(x,\omega)\in\mathbb R^{2d}$, whereas $\gamma\in L^2(\mathbb R^d)$ is the so-called {\em reconstruction (synthesis) window}.
\newline
Formula \eqref{TFrep} looks as a representation of the signal $f$ as a continuous superposition of time-frequency shifts of the window $\gamma$ with \lq\lq coefficients\rq\rq provided by the STFT of $f$ with respect to the window $g$. This is a fundamental counterpart of the Fourier inversion formula (which is just a {\em frequency} representation of a signal $f$) where the involved elementary functions $M_\omega T_x\gamma$ are in $L^2(\mathbb R^d)$. 

\smallskip
Due expecially to numerical purposes, one of the main task of Time Frequency Analysis is replacing the {\em continuous} representation in \eqref{TFrep} by an analogous {\em discrete} representation in terms of  countably many   shifts $M_{\beta k}T_{\alpha h}\gamma$, for $(h,k)\in\mathbb Z^d\times\mathbb Z^d$, where the STFT $V_gf$ is sampled at the {\em discrete lattice} made by points $(\alpha h, \beta k)$, with the {\em width parameters} $\alpha, \beta>0$, that is
\begin{equation}\label{discr_exp}
	f=\sum\limits_{h,k\in\mathbb Z^d}(f, M_{\beta k}T_{\alpha h}g)M_{\beta k}T_{\alpha h}\gamma\,.
\end{equation}
Differently from the continuous case where $g$ and $\gamma$ do not need to satisfy any condition unless $(g,\gamma)\neq 0$, in order to get the discrete representation formula \eqref{discr_exp} the window functions $g$ and $\gamma$ cannot be chosen arbitrarily. First, one needs that the analysis window $g$ and the width parameters $\alpha$, $\beta$ are taken such that $\sum\limits_{h,k\in\mathbb Z^d}\vert(f, M_{\beta k}T_{\alpha h}g)\vert^2$ is equivalent to $\Vert f\Vert^2_{L^2}$; under such assumption on $g$, $\alpha$, $\beta$, it is possible to find a proper reconstruction window $\gamma$ (called a {\em dual window} of $g$), obeying the same conditions as $g$ and making \eqref{discr_exp} to be satisfied, see \cite[Proposition 5.2.1]{GRO1}. This naturally leads to the study of {\em Gabor frames}. For a nontrivial $g\in L^2(\mathbb R^d)\setminus\{0\}$ and given positive numbers $\alpha, \beta$, the system 
\begin{equation}\label{g-alpha-beta}
\mathcal G(g,\alpha,\beta):=\{M_{\beta k}T_{\alpha h} g\,,\,\,(h,k)\in\mathbb Z^d\times\mathbb Z^d\}
\end{equation}
is said to be a {\em Gabor frame} provided that for some constants $A, B>0$:
\begin{equation}\label{Gframe}
A\Vert f\Vert^2_{2}\le\sum\limits_{h,k\in\mathbb Z^d}\vert (f, M_{\beta k}T_{\alpha h}g)\vert^2\le B\Vert f\Vert^2_{2}\,,\quad\forall\, f\in L^2(\mathbb R^d)\,.
\end{equation}
A classical problem in Time Frequency Analysis is to find conditions on triples $(g, \alpha, \beta)$ under which \eqref{g-alpha-beta} is a Gabor frame. The available literature devoted to such a problem is wide; among the others, we recall that a necessary condition for a system \eqref{g-alpha-beta} to be a Gabor frame is that $\alpha\beta\le 1$, see \cite{DAU}, \cite{GRO1} and, for a detailed survey on these results, \cite{Hei07}. Gr\"ochenig - St\"ockler \cite{GS1} are able to show that \eqref{g-alpha-beta} is a Gabor frame if and only if $\alpha\beta<1$, in the case when the window $g\in L^2(\mathbb R^d)$ is a {\em totally positive function of finite type at least two}. We also quote Dai - Sun \cite{Dai-Sun}, where the authors focus on the characteristic function $\chi_{[0,c]}$ of a real interval and study the problem of finding necessary or sufficient conditions on the parameters $\alpha, \beta, c>0$ such that $\mathcal G(\chi_{[0,c]},\alpha,\beta)$ is a Gabor frame.

\smallskip
For every system $\mathcal G(g,\alpha,\beta)$ one can formally define the {\em Gabor operator} as
\begin{equation}\label{Sg}
S_gf:=\sum\limits_{h,k\in\mathbb Z^d}(f, M_{\beta k}T_{\alpha h}g)M_{\beta k}T_{\alpha h}g\,,\quad f\in L^2(\mathbb R^d)\,.
\end{equation}
It is well-known that  $\mathcal G(g,\alpha,\beta)$ is a Gabor frame if and only $S_g$ is invertible as a linear bounded operator on $L^2(\mathbb R^d)$, see \cite{Hei11}, \cite{BogGar2020}.
In this case  there exists $\gamma\in L^2(\mathbb R^d)$ such that $\mathcal G(\gamma,\alpha,\beta)$ is still a Gabor frame, said \textit{dual frame} of $\mathcal G(g,\alpha,\beta)$. Moreover the discrete representation \eqref{discr_exp} holds true for $f\in L^2(\mathbb R^d)$, with unconditional convergence. More precisely one can take $\gamma=S^{-1} g$, see for instance \cite[Proposition 5.2.1]{GRO1}.
The above considerations motivate the interest in studying the $L^2$ boundedness and invertibiliy of the Gabor operator \eqref{Sg}.
\newline
Let us notice that in \cite{BogGar2020} the  Gabor operator \eqref{Sg} is expressed in terms  of a pseudodifferential operator with Kohn-Nirenberg quantization, that is
\begin{equation} \label {GPS}
S_g f(x)=\int e^{2\pi ix\cdot\omega}p(x,\omega)\widehat f(\omega)d\omega\,,	
\end{equation}  
where the symbol $p(x, \omega)$ reads as
\begin{equation}\label{symbG}
p(x,\omega):=\sum_{h,k\in\mathbb Z^d}q(x+\alpha h, \omega+\beta k)\quad\mbox{with}\quad q(x,\omega):=e^{-2\pi ix\cdot\omega} g(x)\overline{\widehat{g}}(\omega).
\end{equation}
The periodicity of the symbol in both $x$ and $\omega$ assures that $p(x,\omega)$ is bounded, but prevents any decay behavior at infinity. Therefore applying to \eqref{GPS} {\em Calder\'{o}n-Vaillancourt Theorem} of $L^2-$boundedness, in \cite{BogGar2020} the authors find sufficient conditions on $g$, and on $\alpha$, $\beta$ such that $S_{g}$ is continuous and invertible on $L^2(\mathbb R^d)$. 
\newline
Concerning the analysis of pseudodifferential operators with partially periodic symbols, we quote \cite{RuzTur}, where pseudodifferential operators with symbols on compact Lie groups are studied; this includes  symbols $p(x,\omega)$ with periodic $x\in \mathbb T^d$ and discrete $\omega \in\mathbb Z^d$. The converse case of $p(x,\omega)$ discrete in $x$ and periodic in $\omega$, which corresponds to the so called \textit{pseudo-difference operators} are studied \cite{BoKiRu20} and \cite{KumMon2023}. Fourier multipliers with periodic symbol $p(\omega)$ were earlier analyzed in \cite{deL65}, \cite{Iga69}. $L^2$ boundedness and invertility of pseudodifferential operators with  symbols $p(x,\omega)$ periodic both in $x$ and $\omega$ are studied in our previous work \cite{GMArx}.
\newline
For $q(x,\omega)$ such that the series in \eqref{symbG} is convergent, we say $p(x,\omega)$ \lq\lq periodization\rq\rq of $q(x,\omega)$.
\newline
Let us observe that the lattice widths $\alpha, \beta$ satisfying the condition in \cite{BogGar2020} are too small, also in one dimensional case, for having any computational utilities. 
\newline
The aim of this paper is to give results of continuity and invertibility for pseudodifferential operators with \textit{periodized} symbol and provide applications for finding conditions on the the window $g$ and the lattice widths $\alpha, \beta$ for a Gabor system $\mathcal G(g, \alpha, \beta)$ to be a Gabor frame. The matter is studied in the framework of generalized $L$-periodic symbols, for a general invertible  matrix $L\in GL(2d)$, that is $T_{L\kappa}p=p$, for all $\kappa \in \mathbb Z^{2d}$. 
In this framework a key role is played by the Poisson summation formula, which is a classical tool of the time-frequency analysis, see \cite[\S 1.4]{GRO1} and the references given there.
\newline
The paper develops as follows. In Section \ref{TOOLS} we introduce and discuss the technical tools needed later on. Particularly we define  the periodic distributions with respect to a generic invertible matrix. Moreover we introduce the periodization of functions and focus on the generalized Poisson formula. In Section \ref{GENSYMB} we introduce the pseudodifferential operators with general $0\le \tau\le 1$ quantization and completely periodic symbol, moreover we recall the related continuity and invertibility results obtained in \cite{GMArx}. Such results are then applied to the case of \textit{periodized} symbols $q_L(x, \omega)=\sum_{\kappa\in \mathbb Z^{2d}}q(z+L\kappa)$, $z=(x,\omega)\in \mathbb R^{2d}$, $q\in L^1(\mathbb R^{2d})$.  In Section \ref{Gf_sec} we apply the previous results for obtaining continuity and invertibility   for a generalized Gabor  operator
\begin{equation}\label{GF oprt0}
S^L_{g,\gamma}u:=\sum\limits_{\kappa\in\mathbb Z^{2d}}(u,\pi_{-L\kappa}g)_{L^2}\,\pi_{-L\kappa}\gamma\,,
\end{equation}
where $\pi_{-L\kappa}$ is the time-frequency shift introduced in the next \eqref{TFS} and $L$ an invertible matrix of order $2d$.
\newline
 Continuity and invertibility, under suitable assumption on $g, \gamma$, are obtained in general \lq\lq time-frequency invariant\rq\rq Banach spaces, see \eqref{PST8}. The sufficient conditions on $L$, $g$ and $\gamma$, found in Section \ref{Gf_sec} for obtaining the continuity and invertibility of the related Gabor operator, are expressed in implicit terms. Thus in Section \ref{diag_sct} we restrict our analysis to the case of diagonal matrix in order to get direct regularity condition on $g$, $\gamma$ and on the \textit{volume}  $\vert \textup{det}L\vert$ of the matrix, assuring the invertibility of the Gabor operator in $L^2(\mathbb R^d)$. Applying this result to the case $\gamma=g$, we get a sufficient condition for the Gabor system $\mathcal G(L,g)$ to be a  Gabor frame and explicit estimates of the frame bounds $A$ and $B$ in \eqref{Gframe}, see Corollary \ref{GFB}.
\newline
 In Appendix \ref{APP} the details of the proof of Proposition \ref{WS-prop:1} are given.
\section {Notations and tools}\label{TOOLS}
In whole the paper we will use the following notations and tools:
\begin{itemize}
\item $\mathbb R^d_0=\mathbb R^d\setminus\{0\}$, $\mathbb Z^d_0=\mathbb Z^d\setminus\{0\}$;
\smallskip

\item $x\cdot \omega=\langle x,\omega\rangle=\sum_{j=1}^d x_j\omega_j$, \quad $x,\omega\in \mathbb R^d$;
\item $(f,g)=\int f(x)\bar g(x)\, dx$ is the inner product in $ L^2(\mathbb R^d)$;
\item $\mathcal F f(\omega)=\hat f(\omega)=\int f(x) e^{-2\pi i x\cdot \omega} \, dx$  is the Fourier transform of $f\in \mathcal S(\mathbb R^d)$, with the well known extension to $u\in \mathcal S'(\mathbb R^d)$;
\item $GL(d)$ is the space of invertible matrices of size $d\times d$.
\end{itemize}

In the following we will use in many cases the matrix in $GL(2d)$ 
\begin{equation}\label{SYMPMAT}
\mathcal J=\left(
\begin{array}{cc}
0&-I\\
I&0
\end{array}
\right),
\end{equation} 
where here and below $0$ and $I$ are respectively the $d\times d$ zero and unit matrices. Through the matrix $\mathcal J$, we define  the \textit{symplectic} form, see \cite[\S 9.4]{GRO1},
\begin{equation}\label{SYMP}
[z_1, z_2]:= \langle z_1, \mathcal J z_2\rangle= x_2\cdot\omega _1-x_1\cdot\omega_2\quad , \quad z_1=(x_1, \omega_1),\, z_2=(x_2, \omega_2) \in \mathbb R^{2d}. 
\end{equation} 

\subsection{Time frequency shifts (tfs)}\label{SubTF}
For $z=(x,\omega)\in \mathbb R^{2d} $ we define the operators:
\begin{align}
&T_x f(t)=f(t-x)  &\text{(translation)};\label{TR}\\
&M_\omega f(t)=e^{2\pi i \omega \cdot t} f(t) &\text{(modulation)},\label{MO}\\
&\pi_z f=M_\omega T_x f  &\text{(time-frequency shift)},\label{TFS}
\end{align}
with suitable extension to distributions in $\mathcal D'(\mathbb R^{d})$.\\
For $u\in\mathcal S'(\mathbb R^{d})$, $z=(x,\omega)\in \mathbb R^{2d}$, the next properties easily follow:
\begin{align}
& T_x M_\omega u=e^{-2\pi i x\cdot\omega}M_\omega T_x u \label{SCAMBIO},\\
&\mathcal F(T_x u)=M_{-x}\mathcal Fu, &&\mathcal F^{-1}(T_x u)=M_x\mathcal F^{-1} u \label{TR1},\\
&\mathcal  F(M_\omega u)=T_\omega\mathcal Fu, &&\mathcal F^{-1}(M_\omega u)=T_{-\omega}\mathcal F^{-1} u\label{MO2},\\
&\mathcal F(\pi_z u)=e^{2\pi i x\cdot \omega}\pi_{\mathcal J^T z}\mathcal Fu; &&\mathcal F^{-1}(\pi_z u)=e^{2\pi i x\cdot \omega}\pi_{\mathcal J z}\mathcal F^{-1} u.\label{TM}
\end{align}

\subsection{Short-time Fourier transform, Modulation spaces}
\begin{definition}\label{stft-def}
For a fixed nontrivial function $g\in L^2(\mathbb R^d)$ the short-time Fourier transform (STFT in the sequel) of $f\in L^2(\mathbb R^d)$ with respect to $g$ is defined as
\begin{equation}\label{stft}
V_gf(z):=(f,\pi_z g)=\int_{\mathbb R^d}f(t)e^{-2\pi i \omega\cdot t}\overline{g(t-x)}dt\,,\quad z=(x,\omega)\in\mathbb R^{2d}\,.
\end{equation}  
\end{definition}
It is well-known that for $f,g\in L^2(\mathbb R^d)$, $V_gf$ is a uniformly continuous function on $\mathbb R^{2d}$, $V_gf\in L^2(\mathbb R^{2d})$ and
\begin{equation*}
\Vert V_gf\Vert_{L^2}=\Vert f\Vert_{L^2}\Vert g\Vert_{L^2}\,,
\end{equation*}
see e.g. \cite[Lemma 3.1.1, Corollary 3.2.2]{GRO1}.
\newline
In the context of the definition above the {\em window function} $g$ is usually chosen to be a smooth cut-off function on a neighborhood of the origin. However, STFT can be usefully studied with much more irregular window functions $g$ or arguments $f$; in principle, $V_gf$ can be suitably extended to be a distribution in $\mathcal S^\prime(\mathbb R^{2d})$ for general $g,f\in\mathcal S^\prime(\mathbb R^d)$, provided the right-hand integral in \eqref{stft} is intended in a formal ``weak'' sense, see again \cite{GRO1}  for details. Moving along this direction, particular interest is devoted to the analysis of the STFT in the case when $g$ or $f$ are taken to belong to {\em weighted modulation spaces}.
\smallskip

 The {\em polynomial weight function} $v$ is defined  for some $s\ge 0$ by
\begin{equation}\label{pol}
v(z)=v_s(z)=(1+\vert z\vert)^{s}\,,\quad\forall\,z\in\mathbb R^{2d}\,.
\end{equation}
A non negative measurable function $m=m(z)$ on $\mathbb R^{2d}$ is said to be a \textit{polynomially moderate} (or \textit{temperate}) weight function if there exists a positive constant $C$ such that
\begin{equation}\label{SOPS1}
m(z_1+z_2)\leq C v(z_1)m(z_2) \quad \text {for all}\, z_1, z_2\in \mathbb R^{2d}.
\end{equation}

For other details about weight functions see \cite[\S 11.1]{GRO1}.

\begin{definition}\label{wms_def}
For $p,q\in[1,+\infty]$, the $m-$weighted space $L^{p,q}_m(\mathbb R^{2d})$ consists of all (Lebesgue) measurable functions $F$ on $\mathbb R^{2d}$, such that the norm
\begin{equation*}
\Vert F\Vert_{L^{p,q}_m}:=\left(\int_{\mathbb R^d}\left(\int_{\mathbb R^d}\vert F(x,\omega)\vert^p m(x,\omega)^p dx\right)^{q/p}d\omega \right)^{1/q}
\end{equation*}
is finite (with the expected modification in the case when at least one among $p$ or $q$ equals $+\infty$).
\end{definition}
\begin{definition}\label{M_def}
For a fixed $g\in\mathcal S(\mathbb R^d)\setminus\{0\}$ and $p,q\in[1,+\infty]$, the $m-$weighted modulation space $M^{p,q}_m(\mathbb R^{d})$ consists of all tempered distributions $f\in\mathcal S^\prime(\mathbb R^{d})$ such that $V_gf\in L^{p,q}_m(\mathbb R^{2d})$, provided with the natural norm
\begin{equation}\label{M-norm}
\Vert f\Vert_{M^{p,q}_m}:=\Vert V_gf\Vert_{L^{p,q}_m}\,.
\end{equation}
\end{definition}
The definition of the space $M^{p,q}_m$ is independent of the choice of the window $g$ in the STFT, and $M^{p,q}_m$ turns out to be a Banach space with respect to a norm \eqref{M-norm} corresponding to any nonzero window $g$ (norms \eqref{M-norm} related to different windows are shown to be equivalent to each other). We definitely address the reader to \cite{GRO1} and the references therein for a thorough study of weighted modulation spaces and their properties. Following \cite{GRO1}, in the case of $p=q$ we denote $M^p_m:=M^{p,p}_m$ and in the {\em unweighted case} of $m(x,\omega)\equiv 1$ we write $M^{p,q}$ and $M^p$ instead of $M^{p,q}$. 
\subsection{Frames in Hilbert spaces} 
Let $\mathcal H$ be a separable Hilbert space, with inner product $(\cdot,\cdot)$ and norm $\Vert x\Vert^2=(x,x)$. 
A sequence $\{x_n\}_{n\in\mathbb N}$ in $\mathcal H$ is a \textit{frame} if there exist two positive constants $A,B$ such that.
\begin{equation}\label{FRAME}
A\Vert x\Vert^2\leq \sum_{n\in \mathbb N}\vert  (x, x_n)\vert^2\leq B \Vert x\Vert^2, \quad x\in \mathcal H.
\end{equation}

\begin{theorem}\label{THEFRAME}
Consider the sequence $\{x_n\}_{n\in\mathbb N}$ in the separable Hilbert space $\mathcal H$ and the operator formally defined by  $S:x\to Sx=\sum_{n\in \mathbb N}(x,x_n)x_n$ , then the following statements are equivalent.
\begin{itemize}
\item [i)] $\{x_n\}_{n\in \mathbb N}$ is a frame;
\item[ii)] there exist two positive constants $A,B$ such that 
\begin{equation}\label{OPINEQ}
A\Vert x\Vert^2\leq (Sx, x)\leq B\Vert x\Vert^2;
\end{equation}
\item[iii)] the operator $S$ is a bounded invertible operator in $\mathcal H$, with bounded inverse.
\end{itemize}
More precisely we have $A=\frac{1}{\Vert S^{-1}\Vert}$,  $B=\Vert S\Vert$, where $\Vert S\Vert$ stands for the operator norm. 
\end{theorem}
For the proof see \cite{Hei11}, \cite{BogGar2020}.

\subsection{Periodic distributions}\label{PDSEC}
Consider now and in the whole paper  $L=(a_{ij})\in GL(n)$, $n\in \mathbb N$.
We say that  a distribution $u\in\mathcal D'(\mathbb R^n)$ is $L-$periodic if
\begin{equation}\label{PD}
T_{L\kappa}u=u\quad,\quad\text{for any}\quad \kappa\in\mathbb Z^n.
\end{equation}
Setting $\mathbb T^n_L:=\mathbb R^n/L\mathbb Z^n$, we identify the set of $L-$periodic distributions with the space $\mathcal D'(\mathbb T^n_L)$ of linear continuous forms on $C^\infty(\mathbb T^n_L)$. Notice that  $\mathcal D'(\mathbb T^n_L)\subset\mathcal S'(\mathbb R^n)$.\\
For any $L-$periodic $u\in \mathcal D'(\mathbb R^n)$ , the distribution $v=u(L \cdot)$ is $I_n-$periodic, where $I_n$ is the unit matrix of size $n$; we say in short that $v$ is $1-${\em periodic}. Thus, up to a linear change of  variable, we may always reduce to consider the case of $1-$ periodic distributions and set $\mathbb T^n=\mathbb T^n_{I_n}$. \\
Any $u\in \mathcal D'(\mathbb T^n_L)$ admits the  Fourier expansion
\begin{equation}\label{FL5}
u= \sum_{\kappa\in\mathbb Z^n} c_{\kappa,L}(u)e^{2\pi i \langle L^{-T} k,\cdot\rangle},
\end{equation}
with the Fourier coefficients
\begin{equation}\label{FL6}
c_{\kappa,L}(u):=c_\kappa(u(L\cdot))= \frac{1}{\vert\textup{det} L\vert}\langle  u, e^{-2\pi i \langle  L^{-T}\kappa, \cdot\rangle}   \rangle_{\mathbb T^n_L}.
\end{equation}
For short in the following we set $c_\kappa(u)=c_{\kappa, L}(u)$.\\
For details about the rigorous calculus of Fourier coefficients in $\mathcal D'(\mathbb T^n)$ see \cite[\S 7.2]{HOR0} and the Appendix A in \cite{GMArx}.\\

Consider $L^p(\mathbb T^n_L)$, $1\leq p <\infty$,  the set of measurable $L$-periodic functions on $\mathbb R^n$ such that $\Vert f\Vert_{L^p(\mathbb T^n_L)}=\int_{L[0,1]^n} \vert f(x)\vert ^p\, dx<\infty$, with obvious modification for the definition of $L^\infty(\mathbb T^n_L)$.\\
Then for $f\in L^1(\mathbb T^n_L)$ the following:\\
\begin{equation}\label{FS1}
f(x)=\sum_{\kappa\in \mathbb Z^n} c_\kappa (f) e^{-2\pi i L^{-T}\kappa\cdot x}
\end{equation}
holds with convergence  in $\mathcal S'(\mathbb R^n)$, and 
\begin{equation}\label{FS2}
c_\kappa(f)=\frac{1}{\vert \textup{det}L\vert }\int_{L[0,1]^n} e^{2\pi i L^{-T}\kappa\cdot x} f(x)\, dx.
\end{equation} 
\begin{remark}\label{LATTICE}
It can be useful to write the Fourier expansion of $u\in\mathcal D'(\mathbb T^n_L)$ in terms of the lattice  $\Lambda=L \mathbb Z^n$:
\begin{equation}\label{FL3}
u=\sum_{\mu\in \Lambda^{\bot}}\hat u(\mu)e^{2\pi i\langle  \mu,\cdot \rangle},
\end{equation}
with \begin{equation}\label{FL4}
\hat u(\mu):=\frac{1}{\textup{vol}(\Lambda)}\langle  u, e^{-2\pi i \langle  \mu, \cdot\rangle}   \rangle_{\mathbb T^n_L}.
\end{equation}
Here $\Lambda^\perp:=L^{-T}\mathbb Z^n$ and  $\textup{vol}(\Lambda):=\vert \textup{det} L\vert=\text{meas} \,(L[0,1]^n)$  are respectively called dual lattice and volume of $\Lambda$.
\end{remark}
\subsection{Periodization and Poisson formula}\label{per_sec}
  For  $f\in L^1(\mathbb R^n)$, $L\in GL(n)$, we call $L$-{\em periodization} of $f$  the function formally defined by
\begin{equation}\label{per_eqt:1}
F_L(x):=\sum\limits_{\kappa\in\mathbb Z^n}f(x+L\kappa)\,,\quad x\in\mathbb R^n\,.
\end{equation} 
Provided that the series in the right-hand side is convergent, at least for almost every $x\in\mathbb R^n$, $F_L$  is by construction $L-$periodic.
\newline
Later on, it will be useful to refer to the partial sums of the series in the right-hand side of \eqref{per_eqt:1}; thus for any non negative integer $h\in\mathbb Z_+$, we set
\begin{equation}\label{per_eqt:3}
F_{L,h}(x):=\sum\limits_{\vert\kappa\vert\le h}f(x+L\kappa)\,,\quad x\in\mathbb R^n\,,
\end{equation}
where, hereafter, the {\em order} $\vert\kappa\vert$ of $\kappa=(k_1,\dots,k_n)\in\mathbb Z^n$ is defined by $\vert\kappa\vert:=\sum\limits_{j=1}^n\vert k_j\vert$.

\smallskip
\noindent
On the consistence of the definition of the $L-$periodization $F_L$, the following result can be proved.
\begin{proposition}\label{per_prop:1}
Let $f\in L^1(\mathbb R^n)$ and $L\in GL(n)$ be given. Then the series in the right-hand side of \eqref{per_eqt:1} is absolutely convergent for a.e. $x\in\mathbb R^n$, $F_L\in L^1(L[0,1)^n)$ and
\begin{equation}\label{per_eqt:3.1}
\int_{L[0,1]^n}F_L(x)dx=\int_{\mathbb R^n} f(x)dx\,.
\end{equation}
Moreover, for any real number $N>n$, we have $F_L(x)(1+\vert x\vert)^{-N}\in L^1(\mathbb R^n)$ and
\begin{equation}\label{per_eqt:4}
\lim\limits_{h\to +\infty}\int\left\vert F_{L,h}(x)-F_L(x)\right\vert (1+\vert x\vert)^{-N}dx=0\,.
\end{equation}
In particular, from the convergence \eqref{per_eqt:4} it follows that the series in the right-hand side of \eqref{per_eqt:1} converges to $F_L$ in $\mathcal S^\prime(\mathbb R^n)$.
\end{proposition} 
We premise the following result to the proof of Proposition \ref{per_prop:1}.
\newline
\begin{lemma}\label{per_lemma:1}
Let $G=G(x)$ be an $L-$periodic function in $\mathbb R^n$ such that $G\in L^1(L[0,1)^n)$. Then $G(x)(1+\vert x\vert)^{-N}\in L^1(\mathbb R^n)$ for any real number $N>n$.
\end{lemma}
\begin{proof}
Without loss of generality let us consider a $1-$periodic  function $G\in L^1(Q)$, with $Q:=[0,1]^n$ the unit cube of $\mathbb R^n$ and $Q_\kappa:=\kappa+Q$, for every $\kappa\in\mathbb Z^n$.  By the countable-additivity of Lebesgue's integral we get, for any $N>n$,
\begin{equation}\label{per_eqt:5}
\int\vert G(x)\vert(1+\vert x\vert)^{-N}dx=\sum\limits_{\kappa\in\mathbb Z^n}\int_{Q_\kappa}\vert G(x)\vert(1+\vert x\vert)^{-N}dx\,.
\end{equation}
On the other hand for $x\in Q_\kappa$ (that is $k_j\le x_j\le k_j+1$ for all $j=1,\dots,n$)
\begin{equation}
\vert x\vert=\sqrt{\sum\limits_{j=1}^n}x_j^2\ge\frac1{(\sqrt 2)^{n-1}}\sum\limits_{j=1}^n\vert x_j\vert\ge \frac1{(\sqrt 2)^{n-1}}(\vert\kappa\vert-n)\,,
\end{equation}
hence splitting the sum in the right-hand side of \eqref{per_eqt:5} we can estimate 
\begin{equation}
\begin{split}
\sum\limits_{\kappa\in\mathbb Z^n}&\int_{Q_\kappa}\vert G(x)\vert(1+\vert x\vert)^{-N}dx \\
&\le \sum\limits_{\vert\kappa\vert\le n}\int_{Q_\kappa}\vert G(x)\vert dx+({\sqrt 2})^{n-1}\sum\limits_{\vert\kappa\vert>n}\frac1{(\vert\kappa\vert-n)^N}\int_{Q_\kappa}\vert G(x)\vert dx\,.
\end{split}
\end{equation}
The $1-$periodicity of $G$ and a change of variable under the integral give
\begin{equation}
\int_{Q_\kappa}\vert G(x)\vert\, dx=\int_{\kappa+Q}\vert G(x+\kappa)\vert\, dx=\Vert G\Vert_{L^1(Q)}\,,\,\,\mbox{for all}\,\,\kappa\in\mathbb Z^n\,,
\end{equation}
hence
\begin{equation}
\int\vert G(x)\vert(1+\vert x\vert)^{-N}dx\le \left(\sum\limits_{\vert\kappa\vert\le n}1+({\sqrt 2})^{n-1}\sum\limits_{\vert\kappa\vert>n}\frac1{(\vert\kappa\vert-n)^N}\right)\Vert G\Vert_{L^1(Q)}\, .
\end{equation}
Since $N>n$, the  series in right-hand side above is convergent. Thus the proof ends by observing that, for some positive constant $C_{n,N}$  depending only on $N$ and the dimension $n$,
\begin{equation}\label{per_eqt:6}
\int\vert G(x)\vert(1+\vert x\vert)^{-N}dx\le C_{n,N}\Vert G\Vert_{L^1(Q)}\,,
\end{equation}
\end {proof}


\smallskip
\noindent
\begin{proof}
[\textit{Proof of Proposition \ref{per_prop:1}.}]  Without loss of generality we set $L=I_n$. Let us consider
\begin{equation}\label{per_eqt:7}
F(x)=\sum\limits_{\kappa\in\mathbb Z^n}f(x+\kappa)\,
\end{equation}
and its partial sums 
\begin{equation}\label{per_eqt:8}
F_h(x):=\sum\limits_{\vert\kappa\vert\le h}f(x+\kappa).
\end{equation}
Take also into account the modules series 
$
\sum\limits_{\kappa\in\mathbb Z^n}\vert f(x+\kappa)\vert\,,
$
and its partial sums
\begin{equation}\label{per_eqt:8.1}
G_h(x):=\sum\limits_{\vert\kappa\vert\le h}\vert f(x+\kappa)\vert\,,\quad\mbox{for}\,\,h\in\mathbb Z_+\,.
\end{equation}
Since $f\in L^1(\mathbb R^n)$, by the countable additivity of Lebesgue's integral and  a change of integration variable, we get
\begin{equation}\label{per_eqt:10}
\sum\limits_{\kappa\in\mathbb Z^n}\int_{Q}\vert f(y+\kappa)\vert dy=\sum\limits_{\kappa\in\mathbb Z^n}\int_{Q_\kappa}\vert f(x)\vert dx=\int\vert f(x)\vert dx<+\infty.
\end{equation}
We can then observe that $\{G_h\}_{h=0}^{+\infty}$ is an increasing sequence and $\int G_h(x)\, dx$ is uniformly bounded by $\Vert f\Vert_{L^1}$ , then by Monotone Convergence Theorem the sequence $\{G_h\}_{h=0}^{+\infty}$ is convergent for a.e. $x\in Q$ and series and integral may be interchanged. Using also \eqref{per_eqt:10} we get
\begin{equation}\label{per_eqt:12}
\int_Q\sum\limits_{\kappa\in\mathbb Z^n}\vert f(x+\kappa)\vert dx=\sum\limits_{\kappa\in\mathbb Z^n}\int_Q\vert f(x+\kappa)\vert dx=\Vert f\Vert_{L^1}\,.
\end{equation}
Let us notice that the convergence of $\{G_h\}_{h=0}^{+\infty}$ a.e. in $Q$ implies the convergence a.e. in $\mathbb R^n$. We can then set
\begin{equation}\label{per_eqt:13}
G(x):=\sum\limits_{\kappa\in\mathbb Z^n}\vert f(x+\kappa)\vert\,,\quad\mbox{for a.e.}\,\,x\in\mathbb R^n\,.
\end{equation} 
Thanks to  \eqref{per_eqt:12}, $G\in L^1(Q)$ and for any integer $h\ge 0$ 
\begin{equation}\label{per_eqt:14}
\vert F_h(x)\vert\le\ G_h(x)\le G(x)\,,\quad\mbox{for a.e.}\,\,x\in\mathbb R^n\,.
\end{equation}
Moreover the series \eqref{per_eqt:7} is absolutely convergent for a.e. $x\in \mathbb R^n$, thus the Dominated Convergence Theorem yields that the point-wise sum
\begin{equation}\label{per_eqt:15}
F(x)=\sum\limits_{\kappa\in\mathbb Z^n}f(x+\kappa)\,,\quad\mbox{for a.e.}\,\,x\in\mathbb R^n
\end{equation} 
is in $L^1(Q)$ and  satisfies:
\begin{equation}\label{per_eqt:16}
\vert F(x)\vert\le G(x)\,,\quad\mbox{for a.e.}\,\,x\in\mathbb R^n\,.
\end{equation}
Formula \eqref{per_eqt:3.1} follows at once after the Dominated Convergence Theorem, similarly to \eqref{per_eqt:12}.
\newline
$F$ and $G$ are $1-$periodic functions in $\mathbb R^n$; thus, in view of Lemma \ref{per_lemma:1}, $F(x)(1+\vert x\vert)^{-N}$ and $G(x)(1+\vert x\vert)^{-N}$ belong to $L^1(\mathbb R^n)$ for all real $N>n$. We already have proven that 
\begin{equation}\label{per_eqt:17}
\lim\limits_{h\to +\infty}\vert F_h(x)-F(x)\vert(1+\vert x\vert)^{-N}=0\,,\quad\mbox{for a.e.}\,\,x\in\mathbb R^n\,;
\end{equation}
moreover \eqref{per_eqt:14}, \eqref{per_eqt:16} yield
\begin{equation}\label{per_eqt:18}
\vert F_h(x)-F(x)\vert(1+\vert x\vert)^{-N}\le 2G(x)(1+\vert x\vert)^{-N}\quad\mbox{for a.e.}\,\,x\in\mathbb R^n\,.
\end{equation}
Therefore applying once again the Dominated Convergence Theorem we obtain 
\begin{equation}\label{per_eqt:19}
\lim\limits_{h\to +\infty}\int\vert F_h(x)-F(x)\vert(1+\vert x\vert)^{-N}dx=0\,,
\end{equation} 
that is \eqref{per_eqt:4}.
\newline
Eventually, the convergence of $\{F_h\}_{h=0}^{+\infty}$ to $F$ in $\mathcal S^\prime(\mathbb R^n)$ easily follows from observing that for any real number $N>n$  and $\varphi\in\mathcal S(\mathbb R^n)$ one has
\begin{equation}\label{per_eqt:20}
\lim\limits_{h\to +\infty}\langle F_h-F,\varphi\rangle=\lim\limits_{h\to +\infty}\int (F_h(x)-F(x))(1+\vert x\vert)^{-N}(1+\vert x\vert)^{N}\varphi(x)dx=0\,,
\end{equation}
because of \eqref{per_eqt:19} and 
\begin{equation}\label{per_eqt:21}
\begin{split}
\left\vert\int (F_h(x)-F(x))(1+\vert x\vert)^{-N}(1+\vert x\vert)^{N}\varphi(x)dx\right\vert\\
\le\Vert(1+\vert\cdot\vert)^N\varphi\Vert_{L^\infty}\Vert(F_h-F)(1+\vert\cdot\vert)^{-N}\Vert_{L^1}\,.
\end{split}
\end{equation}
\end{proof}

\smallskip
\noindent
We are now in the position to state the following weak form of the Poisson summation formula.
\begin{proposition}\label{per_prop:2}
Let $f\in L^1(\mathbb R^n)$ and $L\in GL(n)$ be given. Then
\begin{equation}\label{per_eqt:22}
\sum\limits_{\kappa\in\mathbb Z^n}f(x+L\kappa)=\frac1{\vert det L\vert}\sum\limits_{\kappa\in\mathbb Z^n}\hat f(L^{-T}\kappa)e^{2\pi i L^{-T}\kappa\cdot x}\,,
\end{equation}
where $L^{-T}:=(L^{-1})^T$, the left-hand series is absolutely convergent  for a.e. $x\in\mathbb R^n$ and convergent in $\mathcal S^\prime(\mathbb R^n)$. The right-hand series is convergent in $\mathcal S^\prime(\mathbb R^n)$ and the equality between the series holds true in $\mathcal S^{\prime}(\mathbb R^n)$ 
\end{proposition} 
\begin{proof}
We may still reduce to $L=I_n$. Let $F$ be the $1-$periodic function defined by the series in the left-hand side of \eqref{per_eqt:22}. The absolute convergence a.e. in $\mathbb R^n$ and the convergence in $\mathcal S^\prime(\mathbb R^n)$ directly follow from Proposition \ref{per_prop:1}.
\newline
Arguing as in H\"ormander \cite[Section 7.2]{HOR0} (see also \cite{GMArx}), the Fourier transform of $F$ , regarded as a periodic distribution in $\mathcal S^\prime(\mathbb R^n)$, satisfies
\begin{equation}\label{per_eqt:23}
\hat F=\sum\limits_{\kappa\in\mathbb Z^n}c_{\kappa}(F)\delta_{\kappa}\,,
\end{equation}
where
\begin{equation}\label{per_eqt:24}
c_\kappa(F)=\int_{[0,1]^n} e^{-2\pi i\kappa\cdot x}F(x)dx\,,\quad\kappa\in\mathbb Z^n\,,
\end{equation}
are the Fourier coefficients of $F$ and the series in the right-hand side is convergent to $\hat F$ is $\mathcal S^\prime(\mathbb R^n)$. By inverse Fourier transform of \eqref{per_eqt:23} we have then
\begin{equation}\label{per_eqt:25}
F=\sum\limits_{\kappa\in\mathbb Z^n}c_{\kappa}(F)e^{2\pi i\kappa\cdot x}\,,
\end{equation}
where the Fourier series in the right-hand side above is still convergent in $\mathcal S^\prime(\mathbb R^n)$.
\newline
It remains to compute the form of the Fourier coefficients of $F$, which follows at once from inserting the right-hand side of \eqref{per_eqt:15} in \eqref{per_eqt:24} and applying formula \eqref{per_eqt:3.1} to $e^{-2\pi\kappa\cdot x} f(x)$ instead of $f$, namely for any $\kappa\in\mathbb Z^n$
\begin{equation}\label{per_eqt:26}
\begin{split}
c_\kappa(F)&=\int_{[0,1]^n}\sum\limits_{\ell\in\mathbb Z^n}e^{-2\pi i\kappa\cdot x}f(x+\ell)dx=\int_{[0,1]^n}\sum\limits_{\ell\in\mathbb Z^n}e^{-2\pi i \kappa\cdot (x+\ell)}f(x+\ell)dx\\
&=\int_{\mathbb R^n} e^{-2\pi i\kappa\cdot x}f(x)dx=\hat f(\kappa)\,.
\end{split}
\end{equation}
Thus, replacing the latter in \eqref{per_eqt:25} and $F$ by its definition as the sum of the series in \eqref{per_eqt:15} we end the proof.
\end{proof}
For optimal results concerning the Poisson formula we address to \cite{SW71}, \cite{KR94}, \cite{Gro96}, \cite{GroKop19} .


\section{Pseudodifferential operators}\label{GENSYMB}
We say pseudodifferential operator with $0\leq \tau\leq 1$ quantization and symbol $p(z)=p(x,\omega) \in \mathcal S'(\mathbb R^{2d})$,  the operator acting from $\mathcal S(\mathbb R^{d})$ to $\mathcal S'(\mathbb R^{d})$ defined  by
\begin{equation}\label{PS1}
\textup{Op}_{\tau}(p)u(x):=\int_{\mathbb R^d_\omega}\int_{\mathbb R^d_y}e^{2\pi i (x-y)\cdot \omega}p\left((1-\tau) x+\tau y, \omega\right) u(y)\, dy\,d\omega, \quad u\in \mathcal S(\mathbb R^d).
\end{equation}
The formal integration must be understood in distribution sense.
For the definition and development of pseudodifferential operators see the basic texts  \cite{Shu87}, \cite{Hor94}.

\subsection{Periodic symbols}\label{SPDO}
For the proof of the results in this subsection we refer to \cite[\S 3]{GMArx}.
\begin{definition}\label{TFIdef}
A Banach space $\mathcal S(\mathbb R^{d})\hookrightarrow X\hookrightarrow \mathcal S' (\mathbb R^d)$, with $\mathcal S(\mathbb R^d)$ dense in $X$,  is  \textit{time-frequency shifts invariant} (tfs invariant from now on)  if for some polynomial weight function $v$ and $C>0$
\begin{equation}\label{PST8}
\Vert \pi_{z} u\Vert_X\leq C v(z)\Vert u\Vert_X, \quad u\in X,\quad z\in \mathbb R^{2d}.
\end{equation}
\end{definition}
\begin{example}\label{EXINV} Consider  $p, q\in [1, + \infty)$ and  a polynomially moderated weight function $m$.
\begin{itemize}
\item The $m-$weighted modulation spaces $M^{p,q}_m(\mathbb R^d)$, are tfs invariant, see \cite[Theorem 11.3.5] {GRO1}.

\item  The $m-$weighted Lebesgue space $L^p_m(\mathbb R^{d})$ and $m-$weighted Fourier-Lebesgue space $\mathcal FL^p_m(\mathbb R^d)$, defined as the sets of measurable functions and tempered distributions in $\mathbb R^d$, making finite the norms $\Vert f\Vert_{L^p_{m}}:=\Vert m(\cdot,\omega_0)f\Vert_{L^p}$ and $\Vert f\Vert_{\mathcal FL^p_{m}}=\Vert m(x_0,\cdot) \hat f\Vert_{L^p}$, whatever are $(x_0, \omega_0)\in\mathbb R^{2d}$, are tfs  invariant.
\end{itemize}
\smallskip

In both the examples the positive constants $C$ are  directly obtained from \eqref{SOPS1} and depend only on the weights $m$.
\end{example}

\begin{theorem}[Continuity]\label{TEOPST}
Let $X$ be a tfs  invariant space, $L\in GL(2d)$, $p\in \mathcal D'(\mathbb T^{2d}_L)$. Assume that the Fourier coefficients $c_\kappa(p)$ defined in \eqref{FL6} satisfy,
\begin{equation}\label{TEOPST1}
\Vert c_\kappa(p)\Vert_{\ell^1_{L, v}}:= \sum_{\kappa\in \mathbb Z^{2d}} v\left(\mathcal JL^{-T}\kappa\right)\vert c_\kappa(p)\vert <+\infty.
\end{equation}
 Then  for any $\tau\in [0,1]$ the operator $\textup{Op}_\tau(p)$ is bounded on $X$ and
\begin{equation}\label{TEOPST2}  
\Vert \textup{Op}_\tau(p)\Vert_{\mathcal L(X)}\leq C\Vert c_\kappa(p)\Vert_{\ell^1_{L, v}},
\end{equation}
Where $C$ is the constant in \eqref{PST8}.\\
\end{theorem}
In lattice terms, see Remark \ref{LATTICE}, we can write
\begin{equation}\label{TEOPST3}
\Vert c_\kappa (p)\Vert_{\ell^1_{L,v}}=\sum_{\mu=\in \Lambda^\perp}\vert \hat p(\mu)\vert v(\mathcal J \mu):=\Vert \hat p(\mu)\Vert_{\ell^1_v},
\end{equation}
where $\mu= L^{-T}\kappa$,  $\kappa\in \mathbb Z^{2d}$.
\smallskip

For the study of invertibility condition of pseudodifferential operators we will make use of the well known properties of the von Neumann series in
Banach algebras.  
\begin{theorem}[Invertibility]\label{TEOINV}
Let $X$ be a tfs invariant  space, $L\in GL(2d)$, $p\in \mathcal D' (\mathbb T^{2d}_L)$. Assume that the Fourier coefficients $c_\kappa(p)$, $\kappa\in \mathbb Z^{2d}$, satisfy
\begin{equation}\label{TEOINV1}
c_0(p)\neq 0\quad\text{and}\quad
\sum_{\kappa\in \mathbb Z^{2d}_0}\vert c_\kappa(p)\vert v\left(\pi_{\mathcal J L^{-T}\kappa}\right) <\frac{\vert c_0(p)\vert}{C},
\end{equation} 
where $C$ is the constant  in \eqref{PST8}. Then for any $0\leq \tau\leq 1$ 
\begin{itemize}
\item [i)]
the operator $\textup{Op}_\tau(p)$ is invertible in $\mathcal L(X)$;
\item[ii)]
the  norm in $\mathcal L(X)$ of the inverse operator satisfies the following estimate
\begin{equation}\label{TEOINV2}
\Vert (\textup{Op}_\tau(p))^{-1}\Vert_{\mathcal L(X)}\leq \dfrac{1}{\left(1+C\right)\vert c_0(p)\vert-C\Vert c_k(p)\Vert_{\ell^1_{L,m}}}.
\end{equation} 
\end{itemize}
\end{theorem} 
Notice that, according to the previous estimate, the invertibility of $\textup{Op}_\tau(p)$ is independent of the quantization $\tau$.
\begin{remark}\label{ultradistr_rmk}
{\rm In order to stay in the classical setup of rapidly decreasing functions and tempered distributions, when dealing with modulation spaces, we reduce our previous study to the case when $v(z)$, $z=(x,\omega)$, is a polynomial weight \eqref{pol}. However, weighted modulation spaces can be defined even for more general types of non polynomial weight functions, that are only sub-multiplicative, namely satisfying
\[
v(z_1+z_2)\le v(z_1)v(z_2)\,,\quad\forall\,z_1,\, z_2\in \mathbb R^{2d}\,.
\]
This allows e.g. weight functions which exhibit an exponential growth at infinity. 
\newline
One way to make such an extension is the one indicated by Gr\"ochenig \cite[Chapter 11.4]{GRO1}: it relies on the usage of a space of special windows in STFT and replacing the space of tempered distributions $\mathcal S^\prime(\mathbb R^d)$ by the (topological) dual of the modulation space $M^1_v$, which is shown to include $\mathcal S^\prime(\mathbb R^d)$, for certain non polynomial weight functions $v$. 
\newline
An alternative approach is the one resorting to the Bj\"orck's theory of ultradistributions \cite{Bjo66}, where the modulation spaces are recovered as subspaces of ultradistributions under suitable Gelfand-Shilov type growth conditions \cite{GelShi68}. Along this second approach, Dimovski et al. \cite{DimPilPraTeo-2019} introduced a notion of translation-modulation shift invariant spaces, generalizing to the framework of ultradistributions the notion of time-frequency shift invariant spaces considered in the present paper, see the definition above. It is likely expected that our main results in Theorem \ref{TEOPST} and Theorem \ref{TEOINV} could be extended to the case of non polynomial weight functions, by working in the more general setting of \lq\lq tempered\rq\rq ultradistristributions introduced in \cite{DimPilPraTeo-2019}, instead of standard tempered distributions in $\mathcal S^\prime(\mathbb R^d)$.} 
\end{remark}


\subsection{Continuity and invertibility}
In the following the $L-$periodization procedure described above will be applied to a general symbol $q(z)\in L^1(\mathbb R^{2d})$, with variables $z=(x,\omega)$, for an arbitrarily given invertible matrix $L\in GL(2d)$. In view of \eqref{per_eqt:1}, the $L-${\em periodized symbol} of $q$ is defined to be the $L-$periodic symbol
\begin{equation}\label{LPER3}
q_{L}(z):=\sum\limits_{\kappa\in\mathbb Z^{2d}}q(z+L\kappa)\,,\quad\mbox{for a.e.}\,\,\,z\in\mathbb R^{2n}\,,
\end{equation} 
with the absolute convergence of the right-hand series almost everywhere in $\mathbb R^{2n}$, and where $q_{L}(z)\in L^1(\mathbb T^{2d}_L)$ satisfies
\begin{equation}\label{LPER4}
\int_{\mathbb R^{2d}}q(z)dz=\int_{L[0.1]^{2d}}q_{L}(z)dz\,,
\end{equation}
in view of Proposition \ref{per_prop:1}.
\newline
All the results of Section \ref{SPDO} apply to $q_{L}(z)$, since it is a $L-$periodic symbol. The relation between the Fourier coefficients $\left\{c_{\kappa}(q_{L})\right\}_{\kappa\in\mathbb Z^{2d}}$ of $q_{L}$ and the Fourier transform of the original symbol $q(z)\in L^1(\mathbb R^{2d})$ can be easily derived as follows. For any $\kappa\in\mathbb Z^{2d}$, we compute
\begin{equation}
\begin{split}
c_{\kappa}(q_{L})&=\frac1{\vert{\rm det}L\vert}\int_{L[0,1]^{2d}}q_{L}(z)e^{-2\pi i\langle L^{-T}\kappa, z\rangle} dz\\
&=\frac1{\vert{\rm det}L\vert}\int_{L[0,1]^{2d}}\sum\limits_{\ell\in\mathbb Z^{2d}}q(z+L\ell)e^{-2\pi i\langle L^{-T}\kappa, z\rangle} dz\\
&=\frac1{\vert{\rm det}L\vert}\int_{L[0,1]^{2d}}\sum\limits_{\ell\in\mathbb Z^{2d}}q(z+L\ell)e^{-2\pi i\langle L^{-T}\kappa, z+L\ell\rangle} dz\,,
\end{split}
\end{equation}  
where in the last equality above we have exploited that $e^{-2\pi i\langle L^{-T}\kappa,L\ell\rangle}=e^{2\pi i\langle\kappa,\ell\rangle}=1$ since $\langle\kappa,\ell\rangle\in\mathbb Z$. We then apply once again Proposition \ref{per_prop:1} to $q(z)e^{-2\pi i\langle L^{-T}\kappa, z\rangle}$ to simplify $c_{\kappa}(q_{L})$ as
\begin{equation}\label{LPER5}
\begin{split}
c_{\kappa}(q_{L})&=\frac1{\vert{\rm det}L\vert}\int_{\mathbb R^{2d}}q(z)e^{-2\pi i\langle L^{-T}\kappa, z\rangle} dz=\frac1{\vert{\rm det}L\vert}\hat q(L^{-T}\kappa)\,,
\end{split}
\end{equation}
where
\begin{equation*}
\hat q(\theta)=\int_{\mathbb R^{2d}}q(z)e^{-2\pi i\langle\theta, z\rangle}dz\,,\quad \mbox{for}\,\,\,\theta\in\mathbb R^{2d}\,, 
\end{equation*}
is the  Fourier transform of $q(z)$.
\begin{remark}\label{REMLPER2}
It is worth noticing that
\begin{equation*}
c_{0}(q_{L})=\frac1{\vert{\rm det}L\vert}\int_{\mathbb R^{2d}}q(z)dz\,.
\end{equation*}
\end{remark}
\medskip
As announced above, the results collected in Section \ref{SPDO} concerning the boundedness and the invertibility of a pseudodifferential operator with $L-$periodic $\tau-$symbol $p(z)$ can be translated into corresponding results on the operator ${\rm Op}_\tau(q_{L})$ related to the $L-$ periodization of a symbol $q(z)\in L^1(\mathbb R^{2d})$; plugging the previously found explicit form of the Fourier coefficients of $q_{L}$ into the sufficient condition of boundedness \eqref{TEOPST1}, provided by Theorem \ref{TEOPST}, easily leads to the following sufficient condition of      
boundedness ${\rm Op}_\tau(q_{L})$, for arbitrary $\tau\in[0,1]$.
\begin{proposition}\label{PROPLPER1}
Let $X$ be a time-frequency shifts invariant space, $L\in GL(2d)$, $q(z)\in L^1(\mathbb R^{2d})$. Assume that the Fourier transform $\hat q$ of $q$ makes satisfied the following condition
\begin{equation}\label{PROPLPER1.1}
\sigma_{L,v}(q):=\sum_{\kappa\in \mathbb Z^{2d}} v\left(\mathcal JL^{-T}\kappa\right)\vert\hat q(L^{-T}\kappa)\vert <+\infty.
\end{equation}
Then  for any $\tau\in [0,1]$ the operator $\textup{Op}_\tau(q_{L})$, being $q_{L}$ the $L-$periodized of $q$ defined in \eqref{LPER3}, is bounded on $X$ and
\begin{equation}\label{PROPLPER1.2}  
\Vert \textup{Op}_\tau(q_{L})\Vert_{\mathcal L(X)}\leq \frac{C}{\vert{\rm det}L\vert}\sigma_{L,v}(q),
\end{equation}
where $C$ is the constant in \eqref{PST8}.
\end{proposition}
\begin{proof}
From \eqref{LPER5}, one easily finds for $\Vert c_\kappa(q_{L})\Vert_{\ell^1_{L,v}}$ in \eqref{TEOPST1} the expression 
\[
\Vert c_\kappa(q_{L})\Vert_{\ell^1_{L,v}}=\frac1{\vert{\rm det}L\vert}\sigma_{L,v}(q)\,.
\]
The proof follows at once from the application of Thorem \ref{TEOPST} to $q_{L}$.
\end{proof}
Concerning the invertibility of $\textup{Op}_\tau(q_{L})$ as an operator in $\mathcal L(X)$ we still have just to apply to the $L-$periodized symbol of $q\in L^1(\mathbb R^{2d})$ the analogous result concerning the invertibility of a pseudodifferential operator with $L-$periodic symbol $p(z)$, see \S\, \ref{SPDO}. Then we get the following
\begin{proposition}\label{PROPLPER2}
Let $X$ be a time-frequency shifts invariant space, $L\in GL(2d)$, $q(z)\in L^1(\mathbb R^{2d})$. If $\displaystyle\int_{\mathbb R^{2d}}q(z)dz\neq 0$ and the Fourier transform $\hat q$ satisfies
\begin{equation}\label{PROPLPER2.1}
\sum_{\kappa\in \mathbb Z^{2d}_0} v\left(\mathcal J L^{-T}\kappa\right)\vert\hat q(L^{-T}\kappa)\vert < \frac{1}{C}\left\vert\int_{\mathbb R^{2d}}q(z)dz\right\vert\,,
\end{equation}
  then $\textup{Op}_\tau(q_{L})$ is invertible for all $\tau\in[0,1]$ as an operator in $\mathcal L(X)$. The norm of the inverse operator satisfies the following estimate
\begin{equation}\label{PROPLPER2.1.1}
\Vert (\textup{Op}_\tau(q_L))^{-1}\Vert_{\mathcal L(X)}\leq \dfrac{\vert \textup{det}L\vert}{\left(1+C\right) \sigma_0(q)-C \sigma_{L,v}(q)},
\end{equation}
where $\sigma_0(q)= \vert\int_{\mathbb R^{2d}} q(z)\, dz\vert$, $\sigma_{L,v}(q)$ is defined in \eqref{PROPLPER1.1} and $C$ is the constant in \eqref{PST8}.
\end{proposition}

\section{Application to Gabor  operators}\label{Gf_sec}

We say \textit{(generalized) Gabor system} a sequence $\mathcal G(g,L):=\{\pi_{L\kappa}g\}_{\kappa\in \mathbb Z^{2d}}$, 
where $L\in GL(2d)$ and $g$ is a generic measurable function on $\mathbb R^d$. \\
We can associate to a Gabor system the operator
\begin{equation}\label{GF oprt}
S^L_{g,\gamma}u:=\sum\limits_{\kappa\in\mathbb Z^{2d}}(u,\pi_{-L\kappa}g)_{L^2}\,\pi_{-L\kappa}\gamma\,,
\end{equation}
said \textit {Gabor operator} with windows $g, \gamma$.
Here $\gamma=\gamma(t)$, $g=g(t)$ are regular enough to guarantee the convergence of the series in the right-hand side of \eqref{GF oprt} at least in $\mathcal S^\prime(\mathbb R^d)$, whenever $u\in\mathcal S(\mathbb R^d)$.
By reducing $L$ to a diagonal matrix
\begin{equation}
L=\left(\begin{array}{cc}
\alpha I & 0\\
0 & \beta I
\end{array}\right)\,,
\end{equation}
$\alpha$, $\beta \in \mathbb R_+$ and setting $\gamma=g$, we obtain the Gabor system and operator in classical terms, see e.g.  \cite{GRO1}, \cite{Gro14}, \cite{Hei11}, \cite{BogGar2020} and the reference therein. 

\noindent
We take now $\gamma=\gamma(t)$ and $g=g(t)$ such that 
\begin{equation}\label{Gf_eqt:0}
\gamma\in L^1(\mathbb R^d)\quad\mbox{and}\quad\hat g\in L^1(\mathbb R^d)
\end{equation}
and let $q(x,\omega)$ be the symbol defined by 
\begin{equation}\label{Gf_eqt:1}
q(x,\omega):=e^{-2\pi i x\cdot\omega}(\gamma\otimes\bar{\hat g})(x,\omega)=e^{-2\pi i x\cdot\omega}\gamma(x)\bar{\hat g}(\omega)\,,\quad(x,\omega)\in\mathbb R^{2d}\,,
\end{equation}
so that $q\in L^1(\mathbb R^{2d})$.
\newline
For a given matrix $L\in GL(2d)$, we define the $L-$periodic symbol 
\begin{equation}\label{Gf_eqt:2}
q_L(x,\omega)=\sum\limits_{\kappa\in\mathbb Z^{2d}}q(x+I_1 L\kappa,\omega+I_2 L\kappa)\,,\quad (x,\omega)\in\mathbb R^{2d}\,,
\end{equation}
where $I_1$ and $I_2$ are the $d\times 2d$ matrices defined block-wise as
\begin{equation}\label{Gf_eqt:3}
I_1:=(I\,\,0)\,,\qquad I_2:=(0\,\,I)\,,
\end{equation}
and $I$ and $0$ are respectively the $d\times d$ unit and zero matrices. $q_L(x,\omega)$ is just the $L-$periodization of $q(x,\omega)$, see \eqref{LPER3}.
\newline
Following Boggiatto-Garello \cite{BogGar2020}, we consider the pseudodifferential operator $q_L(\cdot,D)$ with Kohn-Nirenberg quantization (i.e. $\tau=0$) and symbol $q_L(x,\omega)$. $q_L(\cdot,D)$ associates to any rapidly decreasing function $u\in\mathcal S(\mathbb R^d)$ the tempered distribution defined by the formal integral
\begin{equation}\label{Gf_eqt:4}
q_L(x,D)u(x)=\int e^{2\pi i x\cdot\omega}q_L(x,\omega)\hat u(\omega)d\omega\,,
\end{equation}
namely
\begin{equation}\label{Gf_eqt:5}
\begin{split}
\langle q_L(\cdot,D)u,\varphi\rangle:&=\iint e^{2\pi i x\cdot\omega}q_L(x,\omega)\varphi(x)\hat u(\omega)dx d\omega\\
&=\langle q_L(x,\omega),e^{2\pi ix\cdot\omega}(\varphi\otimes\hat u)(x,\omega)\rangle\,,\quad\forall\,\varphi\in\mathcal S(\mathbb R^d)\,.
\end{split}
\end{equation} 
Substituting \eqref{Gf_eqt:2} in \eqref{Gf_eqt:5}, since the series in the right-hand side of \eqref{Gf_eqt:2} is convergent in $\mathcal S^\prime(\mathbb R^{2d})$, we compute
\begin{equation}\label{Gf_eqt:6.1}
\begin{split}
\langle &q_L(\cdot,D)u,\varphi\rangle=\langle q_L(x,\omega),e^{2\pi ix\cdot\omega}(\varphi\otimes\hat u)(x,\omega)\rangle\\
&=\sum\limits_{\kappa\in\mathbb Z^{2d}}\langle q(x+I_1 L\kappa,\omega+I_2 L\kappa),e^{2\pi ix\cdot\omega}(\varphi\otimes\hat u)(x,\omega)\rangle\\
&=\sum\limits_{\kappa\in\mathbb Z^{2d}}\langle e^{-2\pi i (x+I_1 L\kappa)\cdot(\omega+I_2 L\kappa)}(\gamma\otimes\bar{\hat g})(x+I_1 L\kappa,\omega+I_2 L\kappa),e^{2\pi ix\cdot\omega}(\varphi\otimes\hat u)(x,\omega)\rangle\\
&=\sum\limits_{\kappa\in\mathbb Z^{2d}}\langle e^{-2\pi i (x\cdot I_2 L\kappa+I_1 L\kappa\cdot\omega+I_1 L\kappa\cdot I_2 L\kappa)}(\gamma\otimes\bar{\hat g})(x+I_1 L\kappa,\omega+I_2 L\kappa),(\varphi\otimes\hat u)(x,\omega)\rangle\\
&=\!\!\!\sum\limits_{\kappa\in\mathbb Z^{2d}}\langle e^{-2\pi i I_1 L\kappa\cdot I_2 L\kappa}(M_{-I_2 L\kappa}T_{-I_1 L\kappa}\gamma)\otimes (M_{-I_1 L\kappa}T_{-I_2 L\kappa}\bar{\hat g}),(\varphi\otimes\hat u)\rangle\\
&=\sum\limits_{\kappa\in\mathbb Z^{2d}}\langle (T_{-I_1 L\kappa}M_{-I_2 L\kappa}\gamma)\otimes (M_{-I_1 L\kappa}T_{-I_2 L\kappa}\bar{\hat g}),(\varphi\otimes\hat u)\rangle\\
&=\sum\limits_{\kappa\in\mathbb Z^{2d}}\langle T_{-I_1 L\kappa}M_{-I_2 L\kappa}\gamma,\varphi\rangle\,\langle M_{-I_1 L\kappa}T_{-I_2 L\kappa}\bar{\hat g},\hat u\rangle\,,\\
\end{split}
\end{equation} 
obtaining
\begin{equation}\label{Gf_eqt:6}
\langle q_L(\cdot,D)u,\varphi\rangle=\sum\limits_{\kappa\in\mathbb Z^{2d}}\langle T_{-I_1 L\kappa}M_{-I_2 L\kappa}\gamma,\varphi\rangle\,\langle M_{-I_1 L\kappa}T_{-I_2 L\kappa}\bar{\hat g},\hat u\rangle\,,
\end{equation}
where all the involved (numerical) series above are convergent and we made use of the fundamental identity $T_xM_\omega=e^{-2\pi ix\cdot\omega}M_\omega T_x$, see \eqref{SCAMBIO}. Let us focus on the second test $\langle M_{-I_1 L\kappa}T_{-I_2 L\kappa}\bar{\hat g},\hat u\rangle$. Using the identities \eqref{TR1}, \eqref{MO2}, $\bar{\hat f}=\hat{\tilde{\bar f}}$, where $\tilde{f}(t):=f(-t)$ is the symmetric of $f=f(t)$ (extended to distributions by duality, as customary) and Parseval's formula, we get
\begin{equation}\label{Gf_eqt:7}
\begin{split}
\langle & M_{-I_1 L\kappa}T_{-I_2 L\kappa}\bar{\hat g},\hat u\rangle=\langle \overline{M_{I_1 L\kappa}T_{-I_2 L\kappa}{\hat g}},\hat u\rangle=(\hat u,M_{I_1 L\kappa}T_{-I_2 L\kappa}{\hat g})_{L^2}\\
&=(\hat u,\widehat{T_{-I_1 L\kappa}M_{-I_2 L\kappa}g})_{L^2}=(u,T_{-I_1 L\kappa}M_{-I_2 L\kappa}g)_{L^2}\,.
\end{split}
\end{equation}
Substituting \eqref{Gf_eqt:7} into \eqref{Gf_eqt:6} and recovering the notation $\pi_{-L\kappa}:=T_{-I_1 L\kappa}M_{-I_2 L\kappa}$, we have found
\begin{equation}\label{Gf_eqt:8}
\langle q_L(\cdot,D)u,\varphi\rangle=\sum\limits_{\kappa\in\mathbb Z^{2d}}(u,\pi_{-L\kappa}g)_{L^2}\,\langle \pi_{-L\kappa}\gamma,\varphi\rangle\,,\quad \forall\,u\,,\varphi\in\mathcal S(\mathbb R^d)\,,
\end{equation} 
with convergent series in the right-hand side, that is (up to a change of sign of the summation index)
\begin{equation}\label{Gf_eqt:9}
q_L(\cdot,D)u=\sum\limits_{\kappa\in\mathbb Z^{2d}}(u,\pi_{L\kappa}g)_{L^2}\pi_{L\kappa}\gamma\,,\quad \forall\,u\in\mathcal S(\mathbb R^d)\,,
\end{equation} 
with series convergent in $\mathcal S^\prime(\mathbb R^d)$. We can then state the following
\begin{proposition}\label{PROPEQUIV}
Consider $g,\gamma$ measurable functions such that $\gamma\in L^1(\mathbb R^d)$, $\hat g\in L^1(\mathbb R^d)$ and $L\in GL(2d)$. Then 
\begin{equation}\label{PREQ1}
S_{g,\gamma}^L u=q_L(\cdot, D)u\quad \forall u\in \mathcal S(\mathbb R^d),
\end{equation}
where the symbol $q_L(x,\omega)$ is defined in \eqref{Gf_eqt:2}.
\end{proposition}
Considering the Fourier transform of $q(x,\omega)$ in \eqref{Gf_eqt:1}, from Fubini's Theorem and Fourier Inversion formula we easily get
\begin{equation}\label{Gf_eqt:9.1}
\hat{q}(\eta,z)=V_g\gamma(\mathcal J(\eta,z))\,\quad\forall\,(\eta,z)\in\mathbb R^{2d}\,,
\end{equation}
where $\mathcal J$ is the matrix defined in \eqref{SYMPMAT}.
\newline
Indeed, for arbitrary $\eta$, $z$ in $\mathbb R^d$ we compute
\begin{equation*}
\begin{split}
\hat{q}(\eta,z)&=\iint e^{-2\pi i(\eta\cdot x+ z\cdot\omega)}q(x,\omega)dx d\omega=\iint e^{-2\pi i(\eta\cdot x+ z\cdot\omega)}e^{-2\pi i x\cdot\omega}\gamma(x)\bar{\hat g}(\omega)dx d\omega\\
&=\int e^{-2\pi i\eta\cdot x}\gamma(x)\left(\int e^{-2\pi i (x+z)\cdot\omega}\bar{\hat g}(\omega)d\omega\right)dx\\
&=\int e^{-2\pi i\eta\cdot x}\gamma(x)\left(\overline{\int e^{2\pi i (x+z)\cdot\omega}\hat g(\omega)d\omega}\right)dx\\
&=\int e^{-2\pi i\eta\cdot x}\gamma(x)\overline{g(x+z)}dx=\int\gamma(x)\overline{M_{\eta}T_{-z}g(x)}dx\\
&=(\gamma, M_\eta T_{-z}g)_{L^2}=V_g\gamma(-z,\eta)=V_g\gamma(\mathcal J(\eta,z))\,\quad\forall\,(\eta,z)\in\mathbb R^{2d}\,.
\end{split}
\end{equation*}
Notice in particular that setting $z=\eta=0$ in \eqref{Gf_eqt:9.1} we get
\begin{equation}\label{Gf_eqt:9.2}
\hat q(0,0)=(\gamma,g)_{L^2}\,.
\end{equation} 
Now we apply Poisson's formula \eqref{per_eqt:22} to $q(x, \omega)$ to get
\begin{equation}\label{Gf_eqt:10}
q_L(x,\omega)=\frac1{\vert det L\vert}\sum\limits_{\kappa\in\mathbb Z^{2d}}\hat q(L^{-T}\kappa)e^{2\pi i L^{-T}\kappa\cdot (x,\omega)}\,,
\end{equation}   
with convergence in $\mathcal S^\prime(\mathbb R^{2d})$ of the series in the right-hand side, and use
\begin{equation}\label{Gf_eqt:11}
\hat q(L^{-T}\kappa)=V_g\gamma(\mathcal J L^{-T}\kappa)\,,\quad\forall\,\kappa\in\mathbb Z^{2d}\,.
\end{equation}
In order to make the subsequent computations, let $B_1$ and $B_2$ be real matrices of size $2d\times d$ such that
\begin{equation}\label{Gf_eqt:12.0}
L^{-1}=B=(B_1\,\,B_2)\,,
\end{equation}
so that for $\kappa\in\mathbb Z^{2d}$ and $(x,\omega)\in\mathbb R^{2d}$
\begin{equation}\label{Gf_eqt:12}
L^{-T}=\left(\begin{array}{c}B_1^T\\ B_2^T\end{array}\right)\quad\mbox{and}\quad L^{-T}\kappa\cdot(x,\omega)=B_1^T\kappa\cdot x+B_2^T\kappa\cdot\omega\,.
\end{equation}
For arbitrary $u,\varphi\in\mathcal S(\mathbb R^d)$ we compute
\begin{equation}\label{Gf_eqt:13}
\begin{split}
\langle &q_L(\cdot,D)u,\varphi\rangle=\langle q_L(x,\omega),e^{2\pi ix\cdot\omega}(\varphi\otimes\hat u)(x,\omega)\rangle\\
&=\frac1{\vert det L\vert}\sum\limits_{\kappa\in\mathbb Z^{2d}}V_g\gamma(\mathcal J L^{-T}\kappa)\iint e^{2\pi i L^{-Y}\kappa\cdot(x,\omega)}e^{2\pi i x\cdot\omega}\varphi(x)\hat u(\omega)dx d\omega\\
&=\frac1{\vert det L\vert}\sum\limits_{\kappa\in\mathbb Z^{2d}}V_g\gamma(\mathcal J L^{-T}\kappa)\iint e^{2\pi i(B_1^T\kappa\cdot x+B_2^T\kappa\cdot\omega)}e^{2\pi x\cdot\omega}\varphi(x)\hat u(\omega)dx d\omega\,,
\end{split}
\end{equation}
where all (numerical) series occurring above are convergent, because of the convergence in $\mathcal S^\prime(\mathbb R^{2d})$ of the distribution series in the right-hand side of \eqref{Gf_eqt:10}. By Fubini's Theorem and Fourier's Inversion formula, the double integrals under the summation above can be recast as
\begin{equation}\label{Gf_eqt:14}
\begin{split}
&\iint e^{2\pi i(B_1^T\kappa\cdot x+B_2^T\kappa\cdot\omega)}e^{2\pi x\cdot\omega}\varphi(x)\hat u(\omega)dx d\omega\\
&=\int e^{2\pi iB_1^T\kappa\cdot x}\left(\int e^{2\pi i(B_2^T\kappa +x)\cdot\omega}\widehat u(\omega)d\omega\right)\varphi(x)dx\\
&=\int e^{2\pi iB_1^T\kappa\cdot x} T_{-B_2^T\kappa}u(x)\varphi(x)dx=\int M_{B_1^T\kappa}T_{-B^T_2\kappa}u(x)\varphi(x)dx\\
&=\langle M_{B_1^T\kappa}T_{-B^T_2\kappa}u,\varphi\rangle=\langle\pi_{\mathcal J L^{-T}\kappa}u,\varphi\rangle\,.
\end{split}
\end{equation}
where $\mathcal J L^{-T}\kappa=(-B_2^T\kappa,B_1^T\kappa)$, in view of \eqref{Gf_eqt:12}. By Substituting \eqref{Gf_eqt:14} into \eqref{Gf_eqt:13} we end up with 
\begin{equation}\label{Gf_eqt:15}
\langle q_L(\cdot,D)u,\varphi\rangle=\frac1{\vert det L\vert}\sum\limits_{\kappa\in\mathbb Z^{2d}}V_g\gamma(\mathcal J L^{-T}\kappa)\langle\pi_{\mathcal J L^{-T}\kappa}u,\varphi\rangle\,,
\end{equation}
that is for every $u\in\mathcal S(\mathbb R^n)$ we have
\begin{equation}\label{Gf_eqt:16}
q_L(\cdot,D)u=\frac1{\vert det L\vert}\sum\limits_{\kappa\in\mathbb Z^{2d}}V_g\gamma(\mathcal J L^{-T}\kappa)\pi_{\mathcal J L^{-T}\kappa}u\,,
\end{equation}
with convergence in $\mathcal S^\prime(\mathbb R^d)$ of the series in the right-hand side.
\newline
Formula \eqref{Gf_eqt:16} is just Janssen's representation of the Gabor frame operator $S^L_{g,\gamma}=q_L(\cdot,D)$, see \cite{GRO1}. Let us notice that both representations \eqref{Gf_eqt:9} and \eqref{Gf_eqt:16} of the Gabor frame operator $S^L_{g,\gamma}$ hold true under the regularity assumption \eqref{Gf_eqt:0} alone on the windows $\gamma, g$. 



\subsection{Weighted Wiener spaces}\label{WSS}
Recall that $Q:=[0,1]^n$ is the {\em unit cube} of $\mathbb R^n$. For $r\in\mathbb Z^n$,  let $\chi_r=\chi_r(x)$ denotes the characteristic function of $Q_r:=Q+r$, that is
\[
\chi_r(x)=\begin{cases}
1\,,\quad\mbox{if}\,\,x\in Q_r\,,\\
0\,,\quad\mbox{otherwise}\,.
\end{cases}
\]
Let $v=v(x)$ be a polynomial weight function as in \eqref{pol}; the following inequality 
\begin{equation}\label{WS-eq:1}
v(x)\le M_v v(r)\,,\quad\forall\,r\in\mathbb Z^n\,,\,\,\forall\,x\in Q_r\,,
\end{equation}   
holds true, where $M_v:=\max\limits_{z\in Q}v(z)$; indeed it is enough to observe that any $x\in Q_r$ can be decomposed as $x=z+r$, for some $z\in Q$, hence sub-multiplicativity yields
\[
v(x)=v(z+r)\le v(z)v(r)\,;
\]
and \eqref{WS-eq:1} follows at once from the definition of $M_v$.
\begin{definition}\label{WS} 
The weighted Wiener space $W(L^1_v)$ is defined to be the class of all functions $f\in L^\infty(\mathbb R^n)$ such that
\begin{equation}\label{WS-norm}
\Vert f\Vert_{W(L^1_v)}:=\sum\limits_{r\in\mathbb Z^n}v(r)\Vert\chi_r f\Vert_{L^\infty}
\end{equation}
is finite.
\end{definition}
It is well-known that $W(L^1_v)$ is a Banach space with respect to the norm \eqref{WS-norm}, see e.g. \cite{GRO1}. Let $W_0(L^1_v)$ denote the (closed) subspace of $W(L^1_v)$ consisting of all its continuous elements.
\begin{remark}\label{WSrmk:1}
{\rm In view of \eqref{WS-eq:1}, it is clear that a norm equivalent to \eqref{WS-norm} could be  obtained by replacing $v(r)$ with the value of $v$ at any point of the cube $Q_r$ different of $r$, for any $r\in\mathbb Z^n$.} 
\end{remark}

\medskip
The following result deals with the sampling of a function $f\in W_0(L^1_v)$ on a discrete lattice $\Lambda:=L\mathbb Z^n$, for a real matrix $L\in GL(n)$.
\begin{proposition}\label{WS-prop:1}
For every real matrix $L\in GL(n)$ a positive constant $C_{L,v}$, depending only on $L$ and the weight function $v$, exists such that for all $f\in W_0(L^1_v)$:
\begin{equation}\label{prop:1-eq:1}
\sum\limits_{\kappa\in\mathbb Z^n}v(L\kappa)\vert f(L\kappa)\vert\le C_{L,v}\Vert f\Vert_{W(L^1_v)}\,.
\end{equation}
Precisely
\begin{equation}\label{prop:1-est:3}
C_{L,v}:=M_v\prod_{j=1}^n\left(\left[\sum\limits_ {i=1}^{n}\left\vert\frac{a^{i,j}}{det L}\right\vert\right]+1\right)\,,
\end{equation}
where $(a^{ij})$ is the cofactor matrix of $L$ and $M_v$ is the constant in \eqref{WS-eq:1}.
\end{proposition}
The proof of Proposition \ref{WS-prop:1} is detailed in  Appendix \ref{APP}.
\smallskip

In the remaining part of this section $X$ is a  time-frequency invariant Banach space as defined in Definition \ref{TFIdef}.
\subsection{Boundedness of  Gabor  operators}\label{bGf_sct}
 As an application of Proposition \ref{PROPLPER1} to the symbol $q_L(x,\omega)$ defined as in \eqref{Gf_eqt:1} we get the following
\begin{proposition}\label{bGf_prop:1}
Let $\gamma=\gamma(t)$ and $g=g(t)$ measurable functions on $\mathbb R^d$ satisfying the assumption \eqref{Gf_eqt:0} and let $L\in GL(2d)$ be such that
\begin{equation}\label{bGf_eqt:1}
\sigma_{L,g,\gamma}:=\sum_{\kappa\in \mathbb Z^{2d}} v\left(\mathcal J L^{-T}\kappa\right)\vert V_g\gamma(\mathcal J L^{-T}\kappa)\vert <+\infty\,.
\end{equation}
Then the Gabor operator $S^L_{g,\gamma}$ estends to a linear bounded operator in $\mathcal L(X)$ and the operator norm of $S^L_{g,\gamma}$ in $\mathcal L(X)$ enjoys the estimate
\begin{equation}\label{bGf_eqt:2}  
\Vert S^L_{g,\gamma}\Vert_{\mathcal L(X)}\leq \frac{C}{\vert{\rm det}L\vert}\sigma_{L,g,\gamma}\,,
\end{equation}
being $C>0$ the constant in \eqref{PST8}.
\end{proposition}
As a consequence of Proposition \ref{bGf_prop:1} we have the following result
\begin{corollary}\label{bGf_cor:1}
Assume that $\gamma, g\in M^1_v$. Then for every $L\in GL(2d)$ the Gabor  operator $S^L_{g,\gamma}$ extends to a linear bounded operator in $\mathcal L(X)$; as such, the operator norm of $S^L_{g,\gamma}$ satisfies the following estimate
\begin{equation}\label{bGf_eqt:3}
\Vert S^L_{g,\gamma}\Vert_{\mathcal L(X)}\leq C\Vert g\Vert_{M^1_v}\Vert\gamma\Vert_{M^1_v}\,,
\end{equation} 
with $C>0$ depending only on the weight $v$.  
\end{corollary}
\begin{proof}
From \cite[Proposition 12.1.4]{GRO1} we know that $g, \gamma\in M^1_v$ implies that $\gamma\in L^1_{u_1}(\mathbb R^d)$ and $\hat g\in L^1_{u_2}(\mathbb R^d)$, where 
\begin{equation}\label{bGf_eqt:4}
u_1(x):=\inf\limits_{\omega}v(x,\omega)\qquad u_2(\omega):=\inf\limits_{x}v(x,\omega)\,.
\end{equation}
Since $v$ (hence $u_1$, $u_2$) are polynomial weights, $u_1\ge 1$ and $u_2\ge 1$.  It follows at once that $\gamma$ and $\hat g$ belong to the space $L^1(\mathbb R^d)$, as required by Proposition \eqref{bGf_prop:1}.
\cite[Proposition 12.1.11]{GRO1} also gives that $V_g\gamma\in W(L^1_v)$, with the estimate
\begin{equation}\label{bGf_eqt:5}
\Vert V_g\gamma\Vert_{W(L^1_v)}\le C\Vert g\Vert_{M^1_v}\Vert\gamma\Vert_{M^1_v}\,,
\end{equation}
with some positive constant $C$ independent of $g$ and $\gamma$. For the sequel, it is useful recalling that
\begin{equation}\label{bGf_eqt:5.1}
\Vert V_g\gamma\Vert_{W(L^1_v)}=\sum\limits_{(r,s)\in\mathbb Z^{2d}}v(r,s)\Vert\chi_{(r,s)}V_g\gamma\Vert_{L^\infty}\,,
\end{equation}
where we set $Q_{(r,s)}:=(r,s)+Q$, being $Q:=[0,1]^{2d}$ the unit cube, and $\chi_{(r,s)}$ is the characteristic function of the cube $Q_{(r,s)}$.
\newline
Since $V_g\gamma$ is also a continuous function on $\mathbb R^{2d}$ (for it is the Fourier transform of $q_L(x,\omega)\in L^1(\mathbb R^{2d})$), applying Proposition \ref{WS-prop:1} with the matrix $\mathcal JL^{-T}$ instead of $L$, we get the convergence of the series in \eqref{bGf_eqt:1}, as required in the statement of Proposition \ref{bGf_prop:1}; moreover for the sum $\sigma_{L,g,\gamma}$ we get the estimate
\begin{equation}\label{bGf_eqt:6}
\sigma_{L,g,\gamma}\le C_{\mathcal J L^{-T},v}\Vert V_g\gamma\Vert_{W(L^1_v)}\,,
\end{equation}   
where \eqref{prop:1-est:3} (with $\mathcal J L^{-T}$ replaced by $L$) provides an explicit (though involved) dependence on $L$ and $v$ of the contant $C_{\mathcal J L^{-T},v}$, which, instead, is independent of $g$ and $\gamma$.
\newline
So we are in the position to apply Proposition \ref{bGf_prop:1} to conclude that $S^L_{g,\gamma}$ extends to a linear bounded operator in $\mathcal L(X)$. Estimate \eqref{bGf_eqt:3} then follows from gathering the preceding estimates \eqref{bGf_eqt:2}, \eqref{bGf_eqt:5}, \eqref{bGf_eqt:6}.  
\end{proof}

\subsection{Invertibility of  Gabor operators}\label{invGf_sct}
As in the preceding section, we take measurable functions $g=g(t)$ and $\gamma=\gamma(t)$ satisfying assumption \eqref{Gf_eqt:0}. We seek conditions on $L$ and the window functions $g$ and $\gamma$ ensuring the invertibility of the Gabor frame operator $S^L_{g,\gamma}$ as a linear bounded operator in $\mathcal L(X)$.
\newline
From stating Proposition \ref{PROPLPER2} for the symbol $q_L(x,\omega)$ defined by \eqref{Gf_eqt:1} through $g$ and $\gamma$ (recall in particular \eqref{Gf_eqt:9.1}, \eqref{Gf_eqt:9.2}) we get the following
\begin{proposition}\label{InvGf_prop:1}
Consider $g$ and $\gamma$ satisfying \eqref{Gf_eqt:0}, $(\gamma,g)_{L^2}\neq 0$ and $L\in GL(2d)$ such that
\begin{equation}\label{InvGf_eqt:2}
\sum_{\kappa\in \mathbb Z^{2d}_0} v(\mathcal J L^{-T}\kappa)\left\vert V_g\gamma\left(\mathcal J L^{-T}\kappa\right)\right\vert< \frac{1}{C}\left\vert (\gamma,g)_{L^2}\right\vert\,,
\end{equation}
where  $C$ is the constant in \eqref{PST8}. Then $S^L_{g,\gamma}$ is invertible as an element of $\mathcal L(X)$. The norm of the inverse operator satisfies
\begin{equation}\label{InvGf_eqt:2.1}
\Vert (S_{g,\gamma}^L)^{-1}\Vert_{\mathcal L(X)}\leq \dfrac{\vert \textup{det}L\vert}{\left(1+C\right) \vert (\gamma, g)_{L^2}\vert-C \sigma_{L,g,\gamma}},
\end{equation}
where $\sigma_{L,g,\gamma}$ is defined in \eqref{bGf_eqt:1} and $C$ is again the constant in \eqref{PST8}.
\end{proposition}
\begin{remark}
{\rm Notice that condition \eqref{bGf_eqt:1}, ensuring that $S^L_{g,\gamma}$ extends to an element of $\mathcal L(X)$, trivially follows from \eqref{InvGf_eqt:2}.}
\end{remark}
\begin{corollary}\label{CORFRAME} Consider $L\in GL(2d)$ and $g\in L^1\setminus\{0\}$, such that $\hat g\in L^1$, which satisfy
\begin{equation}\label{CORFRAME_eqt:1}
\sum_{\kappa\in \mathbb Z^{2d}_0} v(\mathcal J L^{-T}\kappa)\left\vert V_g g\left(\mathcal J L^{-T}\kappa\right)\right\vert< \Vert g  \Vert_{L^2}^2\,,
\end{equation}
then the Gabor system $\mathcal G_L:=\{\pi_{L\kappa}g \}_{\kappa\in\mathbb Z^{2d}}$ is a frame in $L^2$, with possible frame bounds
\begin{equation}\label{CORFRAMEeqt:2}
\begin{array}{l}
A= \frac{2\Vert g\Vert_{L^2}^2 -\sigma _{L,g}}{\vert\textup{det}L\vert};\\
\\
B= \frac{\sigma_{L,g}}{\vert \textup{det}L\vert}< \frac {2}{\vert \textup{det} L\vert}\Vert g\Vert_{L^2}^2,
\end{array}
\end{equation}
where $\sigma_{L,g}:=\sigma_{L,g,g}$.
\end{corollary}

\section{The diagonal case}\label{diag_sct}
With the aim of making explicit the conditions on $g, \gamma$ and $L$ for the invertibility of the Gabor operator $S^L_{g,\gamma}$ we restrict the analysis to the diagonal lattice $\Lambda=L\mathbb Z^{2d}$, that is we assume the matrix $L\in GL(2d)$ to be a diagonal one
\begin{equation}\label{InvGf_eqt:1}
L=diag(\alpha_1,\dots,\alpha_d,\beta_1,\dots,\beta_d)\,,\quad\alpha_j\,,\,\,\beta_j>0\,,\quad\mbox{for}\,\,j=1,\dots,d\,;
\end{equation}
then one computes at once
\begin{equation}\label{InvGf_eqt:1.1}
L^{-T}=diag\left(\frac1{\alpha_1},\dots.\frac1{\alpha_d},\frac1{\beta_1},\dots,\frac1{\beta_d}\right)
\end{equation}
and for every $\kappa=(h,k)\in\mathbb Z^{2d}$ 
\begin{equation*}
\mathcal J L^{-T}\kappa=\left(-\frac{k}{\beta},\frac{h}{\alpha}\right)\,,
\end{equation*}
where for $h=(h_1,\dots,h_d)$, $k=(k_1,\dots, k_d)$,  $\alpha:=(\alpha_1,\dots,\alpha_d)$, $\beta:=(\beta_1,\dots,\beta_d)$, we set
\begin{equation*}
\frac{h}{\alpha}:=\left(\frac{h_1}{\alpha_1},\dots,\frac{h_d}{\alpha_d}\right)\,,\quad \frac{k}{\beta}:=\left(\frac{k_1}{\beta_1},\dots,\frac{k_d}{\beta_d}\right) 
\end{equation*}  
We consider now the case where $\gamma, g\in M^1_v$. Then from \cite[Proposition 12.1.11]{GRO1} we know that $V_g\gamma\in W(L^1_v)$ and it is a continuous function on $\mathbb R^{2d}$. thus in view of Proposition \ref{WS-prop:1} the left-hand side of \eqref{InvGf_eqt:2} can be estimated by the $W(L^1_v)-$norm of $V_g\gamma$. However, since the origin is excluded from the aforementioned summation, by revisiting the arguments used in the proof of Proposition \ref{WS-prop:1} we show that actually we do not need the whole norm $\Vert V_g\gamma\Vert_{W(L^1_v)}$ in order to get an upper bound of the series in the left-hand side of \eqref{InvGf_eqt:2}: a sufficiently great number of terms of the series definining such a norm can be dropped out, provided that the $\alpha_j'$s and $\beta_j'$s are sufficiently small. 
\newline
If we look at the proof of Proposition \ref{WS-prop:1} the key point is to majorize each of the terms $v\left(-\frac{k}{\beta}, \frac{h}{\alpha}\right)\left\vert V_g\gamma\left(-\frac{k}{\beta},\frac{h}{\alpha}\right)\right\vert$ in \eqref{InvGf_eqt:2} by $\Vert\chi_{(r,s)}V_g\gamma\Vert_{L^\infty}$, with a suitable pair $(r,s)\in\mathbb Z^{2d}$ such that the lattice point $\left(-\frac{k}{\beta}, \frac{h}{\alpha}\right)$ belongs to the cube $Q_{(r,s)}$.

\smallskip
The next technical lemma makes better explicit the matter.
\begin{lemma}\label{InvGf_lemma:1}
If $(r,s)\in\mathbb Z^{2d}$ satisfies
\begin{equation}\label{InvGf_eqt:2.1.1}
\vert r_j\vert<\left[\frac1{\alpha_j}\right]\quad\mbox{and}\quad\vert s_j\vert<\left[\frac1{\beta_j}\right]\,,\quad\mbox{for}\,\,j=1,\dots,d\,,
\end{equation}
then the cube $Q_{(r,s)}$ does not contain any point $\left(-\frac{k}{\beta},\frac{h}{\alpha}\right)$ with $(h,k)\neq (0,0)$.
\end{lemma}
\begin{proof}
Consider a pair $(h,k)\in\mathbb Z^{2d}_0$ such that $\left(-\frac{k}{\beta},\frac{h}{\alpha}\right)\in Q_{(r,s)}$; this means that
\begin{equation}\label{InvGf_eqt:3}
r_j\le -\frac{k_j}{\beta_j}<r_j+1\,,\quad\mbox{and}\quad s_j\le \frac{h_j}{\alpha_j}<s_j+1\,,\quad\mbox{for}\,\,j=1,\dots,d\,.
\end{equation}
From our assumption, at least one component of $h$ or $k$ must be nonzero. Assume that is $k_1\neq 0$ and consider first the case where $k_1<0$. Let us set for simplicity $b_1:=\frac1{\beta_1}$ and $\ell_1:=-k_1$ (so that $\ell_1\ge 1$, as it must be an integer number). Then the inequality for $k_1$ in \eqref{InvGf_eqt:3} is restated as $r_1\le b_1\ell_1<r_1+1$ which yields $b_1<r_1+1$ hence $[b_1]+1\le r_1+1$. If otherwise $k_1>0$ then inequality for $k_1$ in \eqref{InvGf_eqt:3} becomes equivalent to $-1-r_1< b_1k_1\le -r_1$ which yields $-r_1\ge b_1$ hence $r_1\le -[b_1]$. In summary, $k_1\neq 0$ implies that $\vert r_1\vert\ge [b_1]$. Since the preceding argument applies whatever $k_j$ or $h_j$ is taken instead of $k_1$ (where $r$ must be replaced by $s$ if the argument is given for some nonzero component of $h$), we conclude that if the cube $Q_{(r,s)}$ contains any point $(h,k)\in\mathbb Z^{2d}_0$ then there exists at least one index $j\in\{1,2,\dots,d\}$ such that
\begin{equation}\label{InvGf_eqt:4}
\vert r_j\vert\ge\left[\frac1{\alpha_j}\right]\qquad\mbox{or}\qquad\vert s_j\vert\ge\left[\frac1{\beta_j}\right].
\end{equation} 
\end{proof}

\smallskip
In view of the preceding lemma, it appears that some of the series terms involved in the definition of the $\Vert V_g\gamma\Vert_{W(L^1_v)}$ see \eqref{bGf_eqt:5.1}, precisely those corresponding to cubes $Q_{(r,s)}$ which ``lower-left" vertex $(r,s)\in\mathbb Z^{2d}$ satisfies
\begin{equation}\label{InvGf_eqt:5}
\vert r_j\vert<\left[\frac1{\alpha}_j\right]\qquad\mbox{and}\qquad\vert s_j\vert<\left[\frac1{\beta_j}\right]\,,\quad\mbox{for}\,\,j=1,\dots,d\,,
\end{equation}
do not contribute to majorize any term $v\left(-\frac{k}{\beta}, \frac{h}{\alpha}\right)\left\vert V_g\gamma\left(-\frac{k}{\beta},\frac{h}{\alpha}\right)\right\vert$ of the series in \eqref{InvGf_eqt:2}. Let us denote by $Q_{(\alpha,\beta)}$ the union of those cubes $Q_{(r,s)}$ whose vertex $(r,s)\in\mathbb Z^{2d}$ satisfies \eqref{InvGf_eqt:5}. Then arguing as above and in view of \eqref{WS-eq:1} we have
\begin{equation}\label{InvGf_eqt:6}
\begin{split}
\sum_{(h,k)\in \mathbb Z^{2d}_0}& v\left(-\frac{k}{\beta}, \frac{h}{\alpha}\right)\left\vert V_g\gamma\left(-\frac{k}{\beta},\frac{h}{\alpha}\right)\right\vert\\
&\le C_{\mathcal J L^{-T},v}\left(\Vert V_g\gamma\Vert_{W(L^1_v)}-\!\!\sum\limits_{(r,s)\in Q_{(\alpha,\beta)}}v(r,s)\Vert\chi_{(r,s)}V_g\gamma\Vert_{L^\infty} \right)\,.
\end{split}
\end{equation}
By replacing $L$ with $\mathcal J L^{-T}$ in \eqref{prop:1-est:3} we obtain
\begin{equation}\label{InvGf_eqt:7}
C_{\mathcal J L^{-T},v}=M_v\prod_{j=1}^d\left([\alpha_j]+1\right)\left([\beta_j]+1\right)\,,
\end{equation}
where $M_v$ is the constant in \eqref{WS-eq:1}.
It is clear that in order to make advantage from the estimate above, at least one of the $\alpha_j,\beta_j$'s must be less or equal than 1; otherwise $\left[\frac1{\alpha_j}\right]=\left[\frac1{\beta_j}\right]=0$ for all $j=1,\dots,d$ and $Q_{(\alpha,\beta)}$ should be empty. Actually, we notice that if we assume $0<\alpha_j\le 1$ and $0<\beta_j\le 1$ for all $j=1,\dots,d$ then we can upper bound $C_{\mathcal J L^{-T},v}$ in \eqref{InvGf_eqt:7} by $4^dM_v$. Since $\sum\limits_{(r,s)\in Q_{(\alpha,\beta)}}v(r,s)\Vert\chi_{(r,s)}V_g\gamma\Vert_{L^\infty}$ can be regarded as a ``partial sum" of the series defining the norm $\Vert V_g\gamma\Vert_{W(L^1_v)}$, if we take all $\alpha_j'$s and $\beta_j'$s sufficiently small, the number of terms of $\sum\limits_{(r,s)\in Q_{(\alpha,\beta)}}\!\!v(r,s)\Vert\chi_{(r,s)}V_g\gamma\Vert_{L^\infty}$ becomes arbitrarily large and the corresponding ``reminder"
$$
\Vert V_g\gamma\Vert_{W(L^1_v)}-\!\!\sum\limits_{(r,s)\in Q_{(\alpha,\beta)}}v(r,s)\Vert\chi_{(r,s)}V_g\gamma\Vert_{L^\infty}
$$
arbitrarily small.
\newline
In particular, since the constant $C$ in the right-hand side of \eqref{InvGf_eqt:2} is also independent of $L$ (and so of $\alpha_j, \beta_j$'s), we can take the $\alpha_j'$s and $\beta_j'$s small enough so as to make the right-hand side of \eqref{InvGf_eqt:6} smaller than $\frac1{C}\vert(\gamma,g)_{L^2}\vert$, as required in \eqref{InvGf_eqt:2}, therefore the following result holds as a corollary of Proposition \ref{InvGf_prop:1}.
\newline
From now on we write $S^{\alpha, \beta}_{g,\gamma}:=S^L_{g,\gamma}$, with $L=diag(\alpha_1, \dots,\alpha_d, \beta_1,\dots, \beta_d)$.
\begin{corollary}\label{InvGf_cor:1}
Let $\gamma, g\in M^1_v$ satisfy $(\gamma, g)_{L^2}\neq 0$ and $\theta:=\max\limits_{j=1,\dots,d}\{\alpha_j, \beta_j\}$. Then there exists $0<\theta_0<1$ depending only on the weight $v$ and $(\gamma, g)_{L^2}$ such that if $\theta\le\theta_0$ then the Gabor operator $S^{\alpha, \beta}_{g,\gamma}$ is invertible as an element of $\mathcal L(X)$.
\end{corollary}
Unfortunately it seems that no explicit upper bound of $\theta$ is available without imposing additional decay estimates on the growth at infinity of $\gamma$ and $g$.

In order to find an explicit upper bound of $\theta$, defined as above, we introduce a suitable decay at infinity of window functions $\gamma$ and $g$. We consider the scale of Banach spaces $L^\infty_\varepsilon(\mathbb R^d)$ defined as in \cite{BogGar2020}.

\begin{definition}\label{defLinf-eps}
	Let $\varepsilon$ be a positive number. We define $L^\infty_{\varepsilon}(\mathbb R^d)$ to be the vector space of distributions $f\in\mathcal S^\prime(\mathbb R^d)$ such that $(1+\vert\cdot\vert)^{d+\varepsilon} f\in L^\infty(\mathbb R^d)$, equipped with its natural norm
	\begin{equation}\label{normLinf-eps}
		\Vert f\Vert_{L^\infty_\varepsilon}:=\Vert (1+\vert\cdot\vert)^{d+\varepsilon}f\Vert_{L^\infty}\,.
	\end{equation}
\end{definition}
\begin{remark}\label{rmkLinf-eps}
	{\rm It is straightforward to check that the norm \eqref{normLinf-eps} turns $L^\infty_{\varepsilon}(\mathbb R^d)$ into a Banach space. Moreover one can check that $L^\infty_{\varepsilon}(\mathbb R^d)$ is included in $L^2(\mathbb R^d)$ with continuous embedding; indeed for all $f\in L^\infty_{\varepsilon}(\mathbb R^d)$ we get at once
	\[
	\Vert f\Vert_{L^2}\le C_{\varepsilon}\Vert f\Vert_{L^\infty_\varepsilon}\,,
	\]
	with $C^2_{\varepsilon}:=\int_{\mathbb R^d}(1+\vert x\vert)^{-2(d+\varepsilon)}dx$.
	Then the Gabor transform is well defined by \eqref{stft} as a standard Lebesgue integral.
	\newline
	It can be also checked that $L^\infty_\varepsilon(\mathbb R^d)$ is included, with continuous imbedding, into the Wiener space $W(L^1)$, see Definition \ref{WS} with $v\equiv 1$.}
\end{remark}
In view of the application of Proposition \ref{InvGf_prop:1}, our main goal here is to find suitable decay conditions on $V_g\gamma(x,\omega)$ in order to provide an explicit bound of the series 
\begin{equation}\label{Vggamma series}
	\sum\limits_{(h,k)\in\mathbb Z^{2d}_0}\left\vert V_g\gamma\left(-\frac{k}{\beta},\frac{h}{\alpha}\right)\right\vert\,,
\end{equation}
by the sizes $\alpha_j$'s, $\beta_j$'s of the lattice associated to the  Gabor operator  $S^{\alpha, \beta}_{g, \gamma}$.
\subsection{Some preparatory results}\label{prep_sect}
Here and in the following we set for shortness $T^d:=\{0,1\}^d$.
\begin{lemma}\label{lemLinf-eps}
	Let $g\in L^\infty_\varepsilon(\mathbb R^d)$ satisfy
	\begin{equation}\label{Linf-eps:1}
		x_jx^\alpha g\in L^{\infty}_\varepsilon(\mathbb R^d)\,,\qquad\forall\,\alpha\in T^d\quad\mbox{and}\quad j=1,\dots,d\,.
	\end{equation}
	Then for every integer $k=1,\dots, d$
	\begin{equation}\label{Linf-eps:2}
		(1+\vert x\vert)^{d+\varepsilon}(1+\vert x_k\vert)^2\prod_{j\neq k}(1+\vert x_j\vert)\vert g(x)\vert\le 2^{d+1}\mathcal H^\varepsilon_{g,k}\,,\quad\forall\,x\in\mathbb R^d\,,
	\end{equation}
	where 
	\begin{equation}\label{Hcal}
		\mathcal H^\varepsilon_{g,k}:=\max\limits_{\alpha\in T^d}\left\{\Vert g\Vert_{L^\infty_\varepsilon}, \Vert x_k x^\alpha g\Vert_{L^\infty_\varepsilon}\right\}\,.
	\end{equation} 
	In particular
	\begin{equation}\label{Linf-eps:3}
		(1+\vert\cdot\vert)^{d+\varepsilon} g\in L^1(\mathbb R^d)\,,
	\end{equation}
and the following estimate holds
	\begin{equation}\label{Linf-eps:3.1}
		\Vert (1+\vert\cdot\vert)^{d+\varepsilon} g\Vert_{L^1}\le d^d 2^{2d+1}\mathcal H^\varepsilon_{g}\,,
	\end{equation}
	with
	\begin{equation}\label{Linf-eps:3.2}
		\mathcal H_g^\varepsilon:=\max\limits_{1\le k\le d}\mathcal H_{g,k}^\varepsilon\,.
	\end{equation}
\end{lemma}
\begin{proof}
	Let us notice that
	\begin{equation}\label{prod}
		\prod_{j=1}^d(1+\vert x_j\vert)=\sum\limits_{\alpha\in T^d}\vert x^\alpha\vert\,.
	\end{equation}
	Assume, without loss of generality, $k=1$; after \eqref{prod}, we compute
	\begin{equation}\label{prod_1}
			(1+\vert x_1\vert)^2 \prod_{j\neq 1}(1+\vert x_j\vert)=\sum\limits_{\substack{\alpha\in T^d\\ \alpha_1=0}}\vert x^\alpha\vert
			+2\sum\limits_{\substack{\alpha\in T^d\\ \alpha_1=1}}\vert x^\alpha\vert+\sum\limits_{\substack{\alpha\in T^d\\ \alpha_1=1}}\vert x_1\vert \vert x^\alpha\vert\,,
	\end{equation}
	hence, by noticing that the number of multi-indices $\alpha\in T^d$ with either $\alpha_1=0$ or $\alpha_1=1$ is $2^{d-1}$, we have
	\[
	\begin{split}
		&(1+\vert x\vert)^{d+\varepsilon}(1+\vert x_1\vert)^2\prod_{j\neq 1}(1+\vert x_j\vert)\vert g(x)\vert=\sum\limits_{\substack{\alpha\in T^d\\ \alpha_1=0}}(1+\vert x\vert)^{d+\varepsilon}\vert x^\alpha g(x)\vert\\
		&+2 \sum\limits_{\substack{\alpha\in T^d\\ \alpha_1=1}}(1+\vert x\vert)^{d+\varepsilon}\vert x^\alpha g(x)\vert+\sum\limits_{\substack{\alpha\in T^d\\ \alpha_1=1}}(1+\vert x\vert)^{d+\varepsilon}\vert x_1 x^\alpha g(x)\vert\le 2^{d+1}\mathcal H^\varepsilon_{g,1}\,.
	\end{split}
	\] 
Notice that in \eqref{Linf-eps:2} we use the weight $(1+\vert x_k\vert)^2\prod_{j\neq k}(1+\vert x_j\vert)$, $1\le k\le d$ instead of the  easier weight $\prod_{j=1}^d(1+\vert x_j\vert)$, since this latter is not enough to get integrability of $(1+\vert x\vert)^{d+\varepsilon} g$ with respect to $x$ variable.
The presence of the power two of one factor $(1+\vert x_k\vert)^2$ in the former weight allows the $x$-integrability of $(1+\vert x\vert)^{d+\varepsilon} g(x)$, for it is enough to provide a bit faster decay at infinity with respect to all of $x$.  
	\newline
	For an arbitrary $x=(x_1,\dots,x_d)\in\mathbb R^d$ let $1\le k\le d$ be the integer number, dependent on $x$, such that
	\begin{equation}\label{max_k}
		\vert x_k\vert=\max\limits_{1\le j\le d}\vert x_j\vert\,.
	\end{equation}
	Then 
	\begin{equation}\label{manipulations}
		\begin{split}
			(1&+\vert x_k\vert)^2\prod_{j\neq k}(1+\vert x_j\vert)\\
			&=(1+\vert x_k\vert)\underbrace{(1+\vert x_k\vert)^{1/d}\dots(1+\vert x_k\vert)^{1/d}}_{d\,\,{\rm times}}\prod_{j\neq k}(1+\vert x_j\vert)\\
			&\ge\prod_{j=1}^{d}(1+\vert x_j\vert)^{1+1/d}
		\end{split}
	\end{equation}
	and from \eqref{Linf-eps:2} we obtain
	\begin{equation}\label{Linf-eps:2.1}
		(1+\vert x\vert)^{d+\varepsilon}\vert g(x)\vert\le 2^{d+1}\mathcal H^\varepsilon_{g}\prod_{j=1}^{d}(1+\vert x_j\vert)^{-1-1/d}\,,\quad\forall\,x\in\mathbb R^d\,.
	\end{equation}
	Then  \eqref{Linf-eps:3} follows as the function
$
	x\mapsto\prod_{j=1}^{d}(1+\vert x_j\vert)^{-1-1/d}
$
	is integrable in $\mathbb R^d$.
	\newline
By integrating    over $\mathbb R^d$ both sides of \eqref{Linf-eps:2.1} and applying Fubini's Theorem, we obtain
	\[
	\int_{\mathbb R^d}\prod_{j=1}^{d}(1+\vert x_j\vert)^{-1-1/d}dx=\left(2\int_0^{+\infty}(1+y)^{-1-1/d}dy\right)^d=2^d d^d\,,
	\]
which proves  \eqref{Linf-eps:3.1}.
\end{proof}

\medskip
\noindent
Let us now prove the following result
\begin{proposition}\label{propLinf-eps}
	For any $\varepsilon>0$, let $\gamma, g\in\mathcal S^\prime(\mathbb R^d)$ satisfy
	\begin{equation}\label{g-gamma}
		\begin{split}
			&x^\beta\partial^\alpha \gamma, x^\beta\partial_{x_j}\partial^\alpha \gamma, x_jx^\beta\partial^\alpha \gamma\in L^\infty_{\varepsilon}(\mathbb R^d)\,,\\
			&x^\beta\partial^\alpha g, x^\beta\partial_{x_j}\partial^\alpha g, x_jx^\beta\partial^\alpha g\in L^\infty_{\varepsilon}(\mathbb R^d)\,,
		\end{split}
	\end{equation}
	for all $\alpha,\beta\in T^d$ and every integer $1\le j\le d$.
	Then, for any $(x,\omega)\in\mathbb R^{2d}$ and any integers $1\le k\le d$, there holds
	\begin{equation}\label{Linf-eps_eqt:2}
		(1+\vert x\vert)^{d+\varepsilon/2}(1+\vert\omega_k\vert)^2\prod_{j\neq k}(1+\vert\omega_j\vert)\vert V_g\gamma(x,\omega)\vert\le 2^{2d+1}d^d\left(1+\frac1{\pi}\right)^{d+1}\!\!\!\!\!\mathcal K^\varepsilon_\gamma\mathcal K^\varepsilon_g
	\end{equation}
	where
	\begin{equation}\label{Linf-eps_eqt:3}
		\begin{split}
			& \mathcal K^\varepsilon_\gamma:=\max\limits_{\alpha,\beta\in T^d\,,\,\,1\le j\le d}\left\{\Vert x^\beta\partial^\alpha \gamma\Vert_{L^\infty_\varepsilon}, \Vert x^\beta\partial_{x_j}\partial^\alpha \gamma\Vert_{L^\infty_\varepsilon}, \Vert x_jx^\beta\partial^\alpha \gamma\Vert_{L^\infty_\varepsilon}\right\}\,;\\
			& \mathcal K^\varepsilon_g:=\max\limits_{\alpha,\beta\in T^d\,,\,\,1\le j\le d}\left\{\Vert x^\beta\partial^\alpha g\Vert_{L^\infty_\varepsilon}, \Vert x^\beta\partial_{x_j}\partial^\alpha g\Vert_{L^\infty_\varepsilon}, \Vert x_jx^\beta\partial^\alpha g\Vert_{L^\infty_\varepsilon}\right\}
		\end{split}
	\end{equation}
\end{proposition} 
\begin{proof}
	Applying Lemma \ref{lemLinf-eps} and the trivial inequality
	\[
	1+\vert x\vert\le (1+\vert t\vert)(1+\vert t-x\vert)\,,
	\] 
	we get
	\begin{equation}\label{Vg1}
		\begin{split}
			(1+\vert x\vert)^{d+\varepsilon}\vert V_g\gamma(x,\omega)\vert\le\int_{\mathbb R^d}(1+\vert t\vert)^{d+\varepsilon}\vert\gamma(t)\vert(1+\vert t-x\vert)^{d+\varepsilon}\vert g(t-x)\vert dt\\
			\le\Vert\gamma\Vert_{L^\infty_{\varepsilon}}\int_{\mathbb R^d}(1+\vert y\vert)^{d+\varepsilon}\vert g(y)\vert dy\le 2^{2d+1}d^d\Vert\gamma\Vert_{L^\infty_{\varepsilon}}\mathcal H^\varepsilon_g 
		\end{split}
	\end{equation}
	Now let us observe that for given $(x,\omega)\in\mathbb R^{2d}$, $V_g\gamma(x,\omega)$ is nothing but the Fourier transform of $\gamma\cdot T_x\overline g$ computed at $\omega$, that is
	$
	V_g\gamma(x,\omega)=\mathcal F\left(\gamma\cdot T_x\overline g\right)(\omega)\,,
$
	see \eqref{stft}. Then for every $\alpha\in T^d$ the identity
	\begin{equation}\label{Vg:id1}
		\omega^\alpha V_g\gamma(x,\omega)=\left(\frac{i}{2\pi}\right)^{\vert\alpha\vert}\mathcal F\left(\partial^\alpha(\gamma\cdot T_x\overline g)\right)(\omega)
	\end{equation} 
	holds true at least in $\mathcal S^\prime(\mathbb R^d_\omega)$ and point-wise in $x\in\mathbb R^d$. Then  by Leibniz's formula applied to the product of the $L^2$-functions $\gamma$ and $T_x\overline g$, whose derivatives $\partial^\alpha\gamma, \partial^\alpha T_x\overline g\equiv T_x\overline{\partial^\alpha g}$ still belong to $L^2(\mathbb R^d)$, we get 
	\begin{equation}\label{Vg:id2}
		\begin{split}
			\omega^\alpha V_g\gamma(x,\omega)&=\left(\frac{i}{2\pi}\right)^{\vert\alpha\vert}\sum\limits_{\beta\le\alpha}\binom{\alpha}{\beta}\mathcal F\left(\gamma^{(\beta)}\cdot T_x\overline{g^{(\alpha-\beta)}}\right)(\omega)\\
			&=\left(\frac{i}{2\pi}\right)^{\vert\alpha\vert}\sum\limits_{\beta\le\alpha}\binom{\alpha}{\beta}V_{g^{(\beta)}}\gamma^{(\alpha-\beta)}(x,\omega)\,.
		\end{split}
	\end{equation}
Here we used the shortcut $h^{(\nu)}:=\partial^{\nu}h$ to write the partial derivatives of $\gamma$ and $g$ of order $\vert\nu\vert$ for any $\nu\in T^d$. 
	\newline
	By \eqref{prod} and repeating the same arguments used to get estimate \eqref{Vg1} with $g^{(\beta)}$ and $\gamma^{(\alpha-\beta)}$ instead of $g$ and $\gamma$ 
we have
	\begin{equation}\label{stima_int}
		\begin{split}
			&(1+\vert x\vert)^{d+\varepsilon}\prod_{j=1}^d(1+\vert\omega_j\vert)\vert V_g\gamma(x,\omega)\vert\\
			&\le\sum\limits_{\alpha\in T^d}\left(\frac{1}{2\pi}\right)^{\vert\alpha\vert}\sum\limits_{\beta\le\alpha}\binom{\alpha}{\beta}(1+\vert x\vert)^{d+\varepsilon}\vert V_{g^{(\beta)}}\gamma^{(\alpha-\beta)}(x,\omega)\vert\\
			&\le 2^{2d+1}d^d \sum\limits_{\alpha\in T^d}\left(\frac{1}{2\pi}\right)^{\vert\alpha\vert}\sum\limits_{\beta\le\alpha}\binom{\alpha}{\beta}\Vert\gamma^{(\beta)}\Vert_{L^\infty_{\varepsilon}} \mathcal H^\varepsilon_{g^{(\alpha-\beta)}}\\
			&\le 2^{2d+1}d^d\mathcal K^\varepsilon_\gamma \mathcal K^\varepsilon_{g}\sum\limits_{\alpha\in T^d}\left(\frac{1}{2\pi}\right)^{\vert\alpha\vert}2^{\vert\alpha\vert}=2^{2d+1}d^d\mathcal K^\varepsilon_\gamma \mathcal K^\varepsilon_{g}\sum\limits_{\alpha\in T^d}\left(\frac{1}{\pi}\right)^{\vert\alpha\vert}\,.
		\end{split}
	\end{equation}
	Eventually, we notice that the number of multi-indices of $T^d$ which have exactly $k$ components equal to one is $\binom{d}{k}$, hence
	\begin{equation}\label{binom1}
		\sum\limits_{\alpha\in T^d}\left(\frac{1}{\pi}\right)^{\vert\alpha\vert}=\sum_{k=0}^d\binom{d}{k}\left(\frac1{\pi}\right)^k=\left(1+\frac1{\pi}\right)^d
	\end{equation}
	and
	\begin{equation}\label{Linf-eps_eq:4}
		(1+\vert x\vert)^{d+\varepsilon}\prod_{j=1}^d(1+\vert\omega_j\vert)\vert V_g\gamma(x,\omega)\vert\le 2^{2d+1}d^d\left(1+\frac1{\pi}\right)^d\mathcal K^\varepsilon_\gamma \mathcal K^\varepsilon_{g}\,.
	\end{equation}
	In order to get the estimate \eqref{Linf-eps_eqt:2}, set $k=1$, without loss of generality, and apply \eqref{prod_1} to $\omega$. 
Then the estimate of $(1+\vert\omega_1\vert)^2\prod_{j\neq 1}(1+\vert\omega_j\vert)\vert V_g\gamma(x,\omega)\vert$ is reduced to the  evaluation of
	\[
	\sum\limits_{\substack{\alpha\in T^d\\ \alpha_1=0}}\vert\omega^\alpha\vert\vert V_g\gamma(x,\omega)\vert\quad\mbox{and}\quad\sum\limits_{\substack{\alpha\in T^d\\ \alpha_1=1}}\vert\omega^\alpha\vert\vert V_g\gamma(x,\omega)\vert\,,
	\]
	which thanks again to \eqref{prod} are already included within $\prod_{j=1}^d(1+\vert\omega_j\vert)\vert V_g\gamma(x,\omega)\vert$ estimated by \eqref{Linf-eps_eq:4}, and 
\begin{equation}
\sum\limits_{\substack{\alpha\in T^d\\ \alpha_1=1}}\vert\omega^\alpha\vert\,\vert\omega_1\vert\vert V_g\gamma(x,\omega)\vert.
\end{equation} 
Here the multiplier of $\vert V_g\gamma(x,\omega)\vert$, under summation over multi-indices $\alpha\in T^d$ such that $\alpha_1=1$, includes an additional factor $\vert\omega_1\vert$, namely $\vert\omega_1\vert$ is involved in $\vert\omega^\alpha\vert\vert\omega_1\vert$ with the power $2$ (as $\alpha_1=1$).
	\newline
	Arguing similarly as we did before to estimate $\vert \omega^\alpha V_g\gamma(x,\omega)\vert$ for any $\alpha\in T^d$, the additional multiplication by $\omega_1$ turns into a further differentiation of $\gamma(t)\cdot T_x\overline g(t)$ with respect to $t_1$ under Fourier transform, see \eqref{Vg:id1}, hence
	\[
	\begin{split}
		\omega_1\omega^\alpha V_g\gamma(x,\omega)&=\left(\frac{i}{2\pi}\right)^{\vert\alpha\vert+1}\mathcal F\left(\partial^\alpha\partial_{1}(\gamma\cdot T_x\overline g)\right)(\omega)\\
		&=\left(\frac{i}{2\pi}\right)^{\vert\alpha\vert+1}\left\{\left[\mathcal F\left(\partial^\alpha(\partial_1\gamma\cdot T_x\overline g)\right)+\mathcal F\left(\partial^\alpha(\gamma\cdot T_x\overline{\partial_1 g})\right)\right](\omega)\right\}\,.
	\end{split}
	\]
	Now we argue on each of the two terms in the right-hand side above as done for $\mathcal F\left(\partial^\alpha(\gamma\cdot T_x\overline g)\right)(\omega)$, with either $\partial_1\gamma$ and $g$ or $\partial_1 g$ and $\gamma$ replacing $\gamma$ and $g$ respectively. Note that in view of \eqref{Linf-eps:2} both $\partial_1\gamma$ and $\partial_1 g$ obey the same regularity constraints as $\gamma$ and $g$ in order have \eqref{stima_int}. In particular, the reason why $\gamma$, and not $\partial_1 g$, plays the role of $g$ in the estimate of $\mathcal F\left(\partial^\alpha(\gamma\cdot T_x\overline{\partial_1 g})\right)(\omega)$ by the same arguments leading to \eqref{stima_int} is because, in view of \eqref{g-gamma} and Lemma \ref{lemLinf-eps}, $(1+\vert\cdot\vert)^{d+\varepsilon}\partial^\alpha\gamma\in L^1(\mathbb R^d)$ whenever $\alpha\in T^d$; the same could not be true for $\partial_1 g$ without adding to \eqref{g-gamma} the further regularity $x_k x^\beta\partial_j\partial^\alpha g\in L^\infty_{\varepsilon}(\mathbb R^d)$ for all $\alpha, \beta\in T^d$ and integers $1\le k,j\le d$. We then  get
	\[
	\begin{split}
		&\sum\limits_{\substack{\alpha\in T^d\\ \alpha_1=1}}\vert\omega_1\vert\vert\omega^\alpha\vert \vert V_g\gamma(x,\omega)\vert\\
		&\le\sum\limits_{\substack{\alpha\in T^d\\ \alpha_1=1}}\left(\frac{1}{2\pi}\right)^{\vert\alpha\vert+1}\left\{\vert\mathcal F\left(\partial^\alpha(\partial_1\gamma\cdot T_x\overline g)\right)(\omega)\vert+\vert\mathcal F\left(\partial^\alpha(\gamma\cdot T_x\overline{\partial_1 g})\right)(\omega)\vert\right\}\\
		&\le 2^{2d+1}d^d\sum\limits_{\substack{\alpha\in T^d\\ \alpha_1=1}}\left(\frac{1}{2\pi}\right)^{\vert\alpha\vert+1}\!\!\!\!\sum\limits_{\beta\le\alpha}\binom{\alpha}{\beta}\left[\Vert\partial_1\gamma^{(\beta)}\Vert_{L^\infty_{\varepsilon}} \mathcal H^\varepsilon_{g^{(\alpha-\beta)}}+\Vert\partial_1 g^{(\beta)}\Vert_{L^\infty_{\varepsilon}} \mathcal H^\varepsilon_{\gamma^{(\alpha-\beta)}}\right]\\
		&\le 2^{2d+1}d^d\sum\limits_{\substack{\alpha\in T^d\\ \alpha_1=1}}\left(\frac{1}{2\pi}\right)^{\vert\alpha\vert+1}2^{\vert\alpha\vert+1}\mathcal K^\varepsilon_\gamma\mathcal K^\varepsilon_g=2^{2d+1}d^d\sum\limits_{\substack{\alpha\in T^d\\ \alpha_1=1}}\left(\frac{1}{\pi}\right)^{\vert\alpha\vert+1}\mathcal K^\varepsilon_\gamma\mathcal K^\varepsilon_g
	\end{split}
	\]
	and analogously to \eqref{binom1},
	\[
	\sum\limits_{\substack{\alpha\in T^d\\ \alpha_1=1}}\left(\frac{1}{\pi}\right)^{\vert\alpha\vert+1}=\left(\frac1{\pi}\right)^2\sum\limits_{\substack{\alpha\in T^d\\ \alpha_1=1}}\left(\frac1{\pi}\right)^{\vert\alpha\vert-1}\!\!=\left(\frac1{\pi}\right)^2\left(1+\frac1{\pi}\right)^{d-1}\,.
	\]
	Thus we obtain
	\begin{equation}\label{Linf-eps_eqt:5}
		\sum\limits_{\substack{\alpha\in T^d\\ \alpha_1=1}}\vert\omega_1\vert\vert\omega^\alpha\vert \vert V_g\gamma(x,\omega)\vert
		\le 2^{2d+1}d^d\frac1{\pi^2}\left(1+\frac1{\pi}\right)^{d-1}\mathcal K^\varepsilon_\gamma\mathcal K^\varepsilon_g\,.
	\end{equation}
	Gathering the preceding estimates we end up with
	\begin{equation}\label{Linf-eps_eqt:6}
		\begin{split}
			&(1+ \vert x\vert)^{d+\varepsilon}(1+\vert\omega_1\vert)^2\prod_{j\neq 1}(1+\vert\omega_j\vert)\vert V_g\gamma(x,\omega)\vert\\
			&\le 2^{2d+1}d^d \left(\sum\limits_{\substack{\alpha\in T^d\\ \alpha_1=0}}\left(\frac{1}{\pi}\right)^{\vert\alpha\vert}+2\sum\limits_{\substack{\alpha\in T^d\\ \alpha_1=1}}\left(\frac{1}{\pi}\right)^{\vert\alpha\vert}+\frac1{\pi}\sum\limits_{\substack{\alpha\in T^d\\ \alpha_1=1}}\left(\frac{1}{\pi}\right)^{\vert\alpha\vert}\!\right)\mathcal K^\varepsilon_\gamma\mathcal K^\varepsilon_g
		\end{split}
	\end{equation}
	On the other hand
	\[
	\begin{split}
		\sum\limits_{\substack{\alpha\in T^d\\ \alpha_1=0}}\left(\frac{1}{\pi}\right)^{\vert\alpha\vert}&+2\sum\limits_{\substack{\alpha\in T^d\\ \alpha_1=1}}\left(\frac{1}{\pi}\right)^{\vert\alpha\vert}+\frac1{\pi}\sum\limits_{\substack{\alpha\in T^d\\ \alpha_1=1}}\left(\frac{1}{\pi}\right)^{\vert\alpha\vert}\\
		&=\sum\limits_{\substack{\alpha\in T^d\\ \alpha_1=0}}\left(\frac{1}{\pi}\right)^{\vert\alpha\vert}+\left(2+\frac1{\pi}\right)\sum\limits_{\substack{\alpha\in T^d\\ \alpha_1=1}}\left(\frac{1}{\pi}\right)^{\vert\alpha\vert}\\
		&=\sum\limits_{\substack{\alpha\in T^d\\ \alpha_1=0}}\left(\frac{1}{\pi}\right)^{\vert\alpha\vert}+\left(2+\frac1{\pi}\right)\frac1{\pi}\sum\limits_{\substack{\alpha\in T^d\\ \alpha_1=1}}\left(\frac{1}{\pi}\right)^{\vert\alpha\vert-1}
	\end{split}
	\]
	and, similar to \eqref{binom1},
	\[
	\sum\limits_{\substack{\alpha\in T^d\\ \alpha_1=0}}\left(\frac{1}{\pi}\right)^{\vert\alpha\vert}=\sum\limits_{\substack{\alpha\in T^d\\ \alpha_1=1}}\left(\frac{1}{\pi}\right)^{\vert\alpha\vert -1}=\left(1+\frac1{\pi}\right)^{d-1}\,,
	\]
	hence
	\[
	\begin{split}
		\sum\limits_{\substack{\alpha\in T^d\\ \alpha_1=0}}&\left(\frac{1}{\pi}\right)^{\vert\alpha\vert}+\left(2+\frac1{\pi}\right)\frac1{\pi}\sum\limits_{\substack{\alpha\in T^d\\ \alpha_1=1}}\left(\frac{1}{\pi}\right)^{\vert\alpha\vert-1}\\
		&=\left(1+\frac1{\pi}\right)^{d-1}\left(1+\left(2+\frac1{\pi}\right)\frac1{\pi}\right)\\
		&=\left(1+\frac1{\pi}\right)^{d-1}\left(1+\frac{2}{\pi}+\frac1{\pi^2}\right)= \left(1+\frac1{\pi}\right)^{d+1}.
	\end{split}
	\]
	Inserting the last into \eqref{Linf-eps_eqt:6} eventually gives
	\begin{equation*}
		\begin{split}
			(1+ &\vert x\vert)^{d+\varepsilon}(1+\vert\omega_1\vert)^2\prod_{j\neq 1}(1+\vert\omega_j\vert)\vert V_g\gamma(x,\omega)\vert\le 2^{2d+1}d^d\left(1+\frac1{\pi}\right)^{d+1}\mathcal K^\varepsilon_\gamma\mathcal K^\varepsilon_g\,,
		\end{split}
	\end{equation*}
	that is \eqref{Linf-eps_eqt:2} for $k=1$.
\end{proof}
\begin{remark}\label{Vggamma_rmk}
{\rm As regards to the estimates \eqref{Linf-eps_eqt:2}, as well as \eqref{Linf-eps_eq:4}, the presence of the weight $(1+\vert x\vert)^{d+\varepsilon}$ in the left-hand side yields a sufficiently fast decay, as $\vert x\vert$ goes to infinity, to make the function $V_g\gamma(x,\omega)$ integrable with respect to $x$. The reason why the estimates \eqref{Linf-eps_eq:4} are not enough to our future purposes is that the weight $\prod\limits_{j=1}^d(1+\vert\omega_j\vert)$ in the left-hand side of \eqref{Linf-eps_eq:4} does not yield sufficiently fast decay, as $\vert\omega\vert$ goes to infinity, to provide $\omega$-integrability of $V_g\gamma(x,\omega)$. Instead, as we already saw in the proof of preceding Lemma \ref{lemLinf-eps} the presence of one factor $(1+\vert\omega_k\vert)$ with power two in the $\omega$-weight appearing in the left-hand side of \eqref{Linf-eps_eqt:2} above allows recovering the $\omega$-integrability of $V_g\gamma(x,\omega)$, so as to make the latter to be an integrable function in $\mathbb R^d_x\times\mathbb R^d_\omega$.
\newline
Repeating all the arguments used in the proof of Lemma \ref{lemLinf-eps} in order to pass from estimate \eqref{Linf-eps:2} to estimate \eqref{Linf-eps:2.1}, we manage to derive from \eqref{Linf-eps_eqt:2}
\begin{equation}\label{Linf-eps_eqt:8}
	\vert V_g\gamma(x,\omega)\vert\le C_d\mathcal K^\varepsilon_\gamma\mathcal K^\varepsilon_g(1+\vert x\vert)^{-d-\varepsilon}\prod\limits_{j=1}^d(1+\vert\omega_j\vert)^{-1-\frac{1}{d}}\,,
\end{equation}
where
\begin{equation}\label{cst}
	C_d:=2^{2d+1}d^d\left(1+\frac1{\pi}\right)^{d+1}\,,
\end{equation}
which provides the wished integrability of $V_g\gamma(x,\omega)$ in $\mathbb R^d_x\times\mathbb R^d_\omega$.

\medskip
\noindent
By the calculations above, we have shown that both functions $g$ and $\gamma$, obeying the assumptions of Proposition \ref{propLinf-eps}, actually belong  to $M^1(\mathbb R^d)$, for the function $g$ above (playing the role of window in the Gabor transform $V_g\gamma$) could be replaced by any function $\varphi\in\mathcal S(\mathbb R^d)$ and the role of $\gamma$ and $g$ in the Gabor transform could be interchanged through the formula $V_\gamma g(x, \omega)=e^{-2\pi i\omega\cdot x}\overline{V_g \gamma(-x,-\omega)}$ (recall that both $g$ and $\gamma$ belong to $L^2(\mathbb R^d)$, cf Remark \ref{rmkLinf-eps}). Therefore, in principle the invertibility of the generalized Gabor operator $S^{\alpha, \beta}_{g,\gamma}$, correponding to any diagonal matrix $L=diag(\alpha_1,\dots,\alpha_d,\beta_1,\dots,\beta_d)$ with positive numbers $\alpha_j, \beta_j$'s, could be reducted to the framework of Corollary \ref{InvGf_cor:1}.  
\newline
However, by making use of the decay estimate \eqref{Linf-eps_eqt:8} of $V_g\gamma(x,\omega)$, we wish to provide an explicit, possibly computable, upper bound of the sizes $\alpha_j$, $\beta_j$'s of the lattice $\alpha\mathbb Z^d\times\beta\mathbb Z^d$, ensuring the invertibility of the related Gabor operator $S^{\alpha, \beta}_{g,\gamma}$. In view of Proposition \ref{InvGf_prop:1}, this consists in providing an explicit upper bound for the sum of the series \eqref{Vggamma series}. }
\end{remark}
A drawback of estimate \eqref{Linf-eps_eqt:8} is the different size and type of decay at infinity of $V_g\gamma(x,\omega)$, with respect to $x$ and to $\omega$.The decay at infinity in $x$ is of \lq\lq homogeneous\rq\rq type, being measured with respect to $\vert x\vert$, with size $d+\varepsilon$, whereas the decay in $\omega$ is measured \lq\lq separately\rq\rq in each coordinate direction $\omega_j$, with size $1+1/d$.
\newline
The lack of symmetry of \eqref{Linf-eps_eqt:8} could be somehow compensated if we assume some additional regularity to the functions $\gamma$ and $g$. The starting point is noticing that whenever $\gamma, g\in L^2(\mathbb R^d)$, from Plancherel's formula and the use of \eqref{SCAMBIO}-\eqref{MO2}, we get (see \cite[Lemma 3.1.1]{GRO1})
\begin{equation*}\label{Gabor Fourier}
	\begin{split}
		V_g\gamma(x,\omega)&=(\gamma,M_\omega T_x g)_{L^2}=(\hat\gamma, \widehat{M_\omega T_x g})_{L^2}=(\hat\gamma, T_\omega M_{-x}\hat g)_{L^2}\\
		&=e^{-2\pi i\omega\cdot x}(\hat\gamma, M_{-x} T_{\omega}\hat g)_{L^2}=e^{-2\pi i\omega\cdot x}V_{\hat g}\hat\gamma(\omega, -x)\,,
	\end{split}
\end{equation*} 
thus in particular
\begin{equation}\label{modulo Gabor tr}
	\vert V_g\gamma(x,\omega)\vert=\vert V_{\hat g}\hat\gamma(\omega, -x)\vert\,.
\end{equation}
The above suggests that if we assume, in addition to \eqref{g-gamma} on $\gamma$ and $g$, that the Fourier transforms of $\gamma$ and $g$ satisfy the same of \eqref{g-gamma}, i.e.
\begin{equation}\label{hat g-gamma}
	\begin{split}
		&\omega^\beta\partial^\alpha \hat\gamma, \omega^\beta\partial_{\omega_j}\partial^\alpha \hat\gamma, \omega_j\omega^\beta\partial^\alpha \hat\gamma\in L^\infty_{\varepsilon}(\mathbb R^d)\,,\\
		&\omega^\beta\partial^\alpha \hat g, \omega^\beta\partial_{\omega_j}\partial^\alpha \hat g, \omega_j\omega^\beta\partial^\alpha \hat g\in L^\infty_{\varepsilon}(\mathbb R^d)\,,
	\end{split}
\end{equation}
then, in view of \eqref{modulo Gabor tr}, the Gabor transform $V_g\gamma(x,\omega)$ will obey the same decay estimates as \eqref{Linf-eps_eqt:8}, where the role of $x$ and $\omega$ is exchanged.
\newline
More precisely, from Proposition \ref{propLinf-eps} we can deduce the following result.
\begin{corollary}
For any $\varepsilon>0$, let $\gamma, g\in\mathcal S^\prime(\mathbb R^d)$ satisfy the assumptions \eqref{g-gamma} and \eqref{hat g-gamma}. Then  
\begin{equation}\label{Linf-eps_eqt:8 sym}
	\begin{split}
	\vert V_g\gamma(x,\omega)\vert\le C_d\mathcal K^{\varepsilon}_{\gamma,\hat\gamma}\mathcal K^{\varepsilon}_{g,\hat g} &(1+\vert x\vert)^{-\frac{d}{2}-\frac{\varepsilon}{2}}(1+\vert \omega\vert)^{-\frac{d}{2}-\frac{\varepsilon}{2}}\\
	&\times \prod\limits_{j=1}^d(1+\vert x_j\vert)^{-\frac12-\frac{1}{2d}}(1+\vert\omega_j\vert)^{-\frac12-\frac{1}{2d}}\,,
	\end{split}
\end{equation}
where $C_d$ is the constant defined in \eqref{cst},
\begin{equation}\label{L-inf-eps_sym eqt}
	\begin{split}
		\mathcal K^{\varepsilon}_{\gamma,\hat\gamma}:=\left(\mathcal K^\varepsilon_{\gamma}\right)^{\frac12}\left(\mathcal K^\varepsilon_{\hat\gamma}\right)^{\frac12}\quad\mbox{and}\quad\mathcal K^{\varepsilon}_{g,\hat g}:=\left(\mathcal K^\varepsilon_{g}\right)^{\frac12}\left(\mathcal K^\varepsilon_{\hat g}\right)^{\frac12}\,,
	\end{split}
\end{equation}  
$\mathcal K^\varepsilon_{\gamma}$ and $\mathcal K^\varepsilon_{g}$ are the constants in \eqref{Linf-eps_eqt:3} and $\mathcal K^\varepsilon_{\hat\gamma}$, $\mathcal K^\varepsilon_{\hat g}$ the same constants with $\hat\gamma$, $\hat g$ instead of $\gamma$, $g$.
\end{corollary} 
\begin{proof}
Starting from \eqref{modulo Gabor tr}, the same arguments used above lead to the following counterpart of \eqref{Linf-eps_eqt:8}
\begin{equation}\label{Linf-eps_eqt:9 hat}
		\vert V_g\gamma(x,\omega)\vert\le C_d\mathcal K^\varepsilon_\gamma\mathcal K^\varepsilon_g(1+\vert \omega\vert)^{-d-\varepsilon}\prod\limits_{j=1}^d(1+\vert x_j\vert)^{-1-\frac{1}{d}}\,,
\end{equation}
where $\mathcal K_{\hat\gamma}^\varepsilon$, $\mathcal K^\varepsilon_{\hat g}$ are defined as in \eqref{Linf-eps_eqt:3} for $\hat\gamma$ and $\hat g$ and $C_d$ is the constant in \eqref{cst}. Then a convex combination of \eqref{Linf-eps_eqt:8} and \eqref{Linf-eps_eqt:9 hat} gives raise to the estimate in \eqref{Linf-eps_eqt:8 sym}.
\end{proof}
\subsection {Existence of Gabor frames} \label{GFEX}
Although the different decay  of $V_g\gamma(x,\omega)$ shown in \eqref{Linf-eps_eqt:8 sym} does not prevent integrability in both $x$ and $\omega$, considering a single decay with respect to the whole variable $(x,\omega)$ will simplify our subsequent calculations. Therefore, we \lq\lq reduce\rq\rq the decay associated to the homogeneous weights $1+\vert x\vert$, $1+\vert\omega\vert$ to the somewhat \lq\lq worser\rq\rq decay corresponding to the weights $\prod\limits_{j=1}^d(1+\vert x_j\vert)$, $\prod\limits_{j=1}^d(1+\vert \omega_j\vert)$. This can be done by the trivial estimate  
\begin{equation}\label{trivial}
	(1+\vert x\vert)^{d+\varepsilon}=\underbrace{(1+\vert x\vert)^{1+\varepsilon/d}\dots(1+\vert x\vert)^{1+\varepsilon/d}}_{d\,\,{\rm times}}\ge\prod\limits_{j=1}^d(1+\vert x_j\vert)^{1+\varepsilon/d}\,.
\end{equation}
Then from \eqref{Linf-eps_eqt:8 sym} we deduce at once
\begin{equation}\label{Linf-eps_eqt:9 hat1}
	\vert V_g\gamma(x,\omega)\vert\le C_d\mathcal K^{\varepsilon}_{\gamma,\hat\gamma}\mathcal K^{\varepsilon}_{g,\hat g}\prod\limits_{j=1}^d(1+\vert x_j\vert)^{-1-\frac{1+\varepsilon}{2d}}(1+\vert\omega_j\vert)^{-1-\frac{1+\varepsilon}{2d}}\,,
\end{equation}
Using \eqref{Linf-eps_eqt:9 hat1} to estimate each of the summands of the series \eqref{Vggamma series}, we get
\begin{equation}\label{Vg sym est 0}
	\begin{split}
		\sum\limits_{(h,k)\in\mathbb Z^{2d}_0}&\left\vert V_g\gamma\left(-\frac{k}{\beta},\frac{h}{\alpha}\right)\right\vert\\
		&\le C_d\mathcal K^\varepsilon_{\gamma, \hat\gamma}\mathcal K^\varepsilon_{g, \hat g}\sum\limits_{(h,k)\in\mathbb Z^{2d}_0}\prod\limits_{j=1}^d\left(1+\left\vert \frac{h_j}{\alpha_j}\right\vert\right)^{-1-\frac{1+\varepsilon}{2d}}\left(1+\left\vert\frac{k_j}{\beta_j}\right\vert\right)^{-1-\frac{1+\varepsilon}{2d}}.
	\end{split}
\end{equation}
In order to go on with an estimate of right-hand side above, it is worth to rearrange the terms in the multi-dimensional sum over $(h,k)\in\mathbb Z^{2d}_0$ in such a way to be reduced to estimate several one-dimensional elementary sums. Let us make a reordering of the summands in $(h,k)$ according to the increasing number of nonzero components of $(h,k)$, that is
\begin{equation}\label{Z split}
	\sum\limits_{(h,k)\in\mathbb Z^{2d}_0}\prod\limits_{j=1}^d\left(1+\left\vert \frac{h_j}{\alpha_j}\right\vert\right)^{-1-\frac{1+\varepsilon}{2d}}\left(1+\left\vert\frac{k_j}{\beta_j}\right\vert\right)^{-1-\frac{1+\varepsilon}{2d}}=\sum\limits_{\ell=1}^{2d}\sum\limits_{(h,k)\in\mathbb Z^{2d}_{(\ell)}}a_{(h,k)}\,.
\end{equation}
In the right-hand side above $a_{(h,k)}$ stands for the general term of the sum over $(h,k)$ in the left-hand side, that is
\begin{equation*}\label{a(hk)}
	a_{(h,k)}:=\prod\limits_{j=1}^d\left(1+\left\vert \frac{h_j}{\alpha_j}\right\vert\right)^{-1-\frac{1+\varepsilon}{2d}}\left(1+\left\vert\frac{k_j}{\beta_j}\right\vert\right)^{-1-\frac{1+\varepsilon}{2d}}\,,
\end{equation*}
and for any integer number $1\le\ell\le 2d$, $\mathbb Z^{2d}_{(\ell)}$ denotes the subset of $\mathbb Z^{2d}_0$ consisting of those multi-integers $(h,k)\in\mathbb Z_0^{2d}$ which have exactly $\ell$ components $h_j$'s and $k_j$'s different from zero. Therefore
\[
\mathbb Z^{2d}_0=\bigcup\limits_{\ell=1}^{2d}\mathbb Z^{2d}_{(\ell)}\qquad\mbox{and}\qquad\mathbb Z^{2d}_{(\ell_1)}\cap\mathbb Z^{2d}_{(\ell_2)}=\emptyset\,\,\,\mbox{for}\,\,\,\ell_1\neq\ell_2\,.
\]
Let us first consider the most simple case of $\sum\limits_{(h,k)\in \mathbb Z^{2d}_{(1)}}a_{(h,k)}$, where $\mathbb Z^{2d}_{(1)}$  is the set of multi-integers $(h,k)$ with all components but one components equal to zero. All multi-integers $(h_1,\underbrace{0,\dots,0}_{d-1},\underbrace{0,\dots,0}_{d})$, for every $h_1\in\mathbb Z\setminus\left\{0\right\}$, will belong to $\mathbb Z^{2d}_{(1)}$, and for the sum of the related terms we get
\begin{equation}\label{1-d sum}
	\sum\limits_{h_1\in\mathbb Z\setminus\{0\}}a_{(h_1,0,\dots,0,0,\dots,0)}=\sum\limits_{h_1\in\mathbb Z\setminus\{0\}}\left(1+\left\vert \frac{h_1}{\alpha_1}\right\vert\right)^{-1-\frac{1+\varepsilon}{2d}}\,,
\end{equation}
where the one-dimensional sum in the right-hand side is estimated by 
\begin{equation}\label{1-d est alpha}
	\begin{split}
		\sum\limits_{h_1\in\mathbb Z\setminus\{0\}}&\left(1+\left\vert \frac{h_1}{\alpha_1}\right\vert\right)^{-1-\frac{1+\varepsilon}{2d}}=2\sum\limits_{h_1=1}^{+\infty}\left(1+\frac{h_1}{\alpha_1}\right)^{-1-\frac{1+\varepsilon}{2d}}\\
		&\le 2\int_0^{+\infty}\left(1+\frac{y}{\alpha_1}\right)^{-1-\frac{1+\varepsilon}{2d}}dy=\frac{4d}{1+\varepsilon}\alpha_1\,.
	\end{split}
\end{equation}
The whole set $\mathbb Z^{2d}_{(1)}$ is  obtained by letting only one among the indices $h_j$'s, $k_j$'s to be different from zero. Thus a similar one-dimensional summation as in \eqref{1-d sum} is involved in $\sum\limits_{(h,k)\in \mathbb Z^{2d}_{(1)}}a_{(h,k)}$, where $h_1$ and $\alpha_1$ will be respectively replaced by $h_2$, $\alpha_2$, ..., $h_d$, $\alpha_d$, $k_1$, $\beta_1$, ..., $k_d$, $\beta_d$, and will be estimated as in \eqref{1-d est alpha}.
\newline
For any integer $j=1,\dots, d$ one computes the same as above with $h_j$, $\alpha_j$ or $k_j$, $\beta_j$ replacing $h_1$, $\alpha_1$, and similarly as before
\begin{equation}\label{1-d est beta}
	\begin{split}
		\sum\limits_{k_j\in\mathbb Z\setminus\{0\}}&\left(1+\left\vert\frac{k_j}{\beta_j}\right\vert\right)^{-1-\frac{1+\varepsilon}{2d}}\le\frac{4d}{1+\varepsilon}\beta_j\,,\quad j=1,\dots,d\,.
	\end{split}
\end{equation}
In summary, we get
\begin{equation}\label{Z1 est}
	\sum\limits_{(h,k)\in\mathbb Z^{2d}_{(1)}}a_{(h,k)}\le\frac{4d}{1+\varepsilon}\sum\limits_{j=1}^d(\alpha_j+\beta_j)\,.
\end{equation}
Let us estimate now the sum
\begin{equation}\label{Z2}
	\sum\limits_{(h,k)\in\mathbb Z^{2d}_{(2)}}a_{(h,k)}\,.
\end{equation}
Here above $\mathbb Z^{2d}_{(2)}$ consists, by definition, of all multi-integers $(h,k)$ where only two components are different from zero. To fix the ideas, let us consider the case where all components of $h$ and $k$, but $h_1$ and $h_2$, are equal to zero. Then the sum of the corresponding $a_{(h,k)}$'s becomes
\[
\begin{split}
	\Sigma_{h_1,h_2}:&=\sum\limits_{h_1, h_2\in\mathbb Z\setminus\{0\}}a_{(h_1,h_2,{\small \underbrace{0,\dots,0}_{d-2},\underbrace{0\dots,0}_{d}})}\\
	&=\!\!\sum\limits_{h_1, h_2\in\mathbb Z\setminus\{0\}}\!\!\!\left(1+\left\vert \frac{h_1}{\alpha_1}\right\vert\right)^{-1-\frac{1+\varepsilon}{2d}}\!\!\left(1+\left\vert \frac{h_2}{\alpha_2}\right\vert\right)^{-1-\frac{1+\varepsilon}{2d}}\,.
\end{split}
\]
By interchanging the sum and the product, we rearrange $\Sigma_{h_1,h_2}$ as follows
\begin{equation*}\label{Sigma h1h2}
	\Sigma_{h_1, h_2}=\sum\limits_{h_1\in\mathbb Z\setminus\{0\}}\left(1+\left\vert \frac{h_1}{\alpha_1}\right\vert\right)^{-1-\frac{1+\varepsilon}{2d}}\!\!\sum\limits_{h_2\in\mathbb Z\setminus\{0\}}\left(1+\left\vert \frac{h_2}{\alpha_2}\right\vert\right)^{-1-\frac{1+\varepsilon}{2d}}
\end{equation*}
and use again \eqref{1-d est alpha} for each of the sums over $h_1$ and $h_2$ above, to get
\begin{equation}\label{est Sigma h1h2}
	\Sigma_{h_1, h_2}\le\left(\frac{4d}{1+\varepsilon}\right)^2\alpha_1\alpha_2\,.
\end{equation}
It is clear that a similar sum as $\Sigma_{h_1,h_2}$ is involved within the sum \eqref{Z2}, corresponding to any way we can select two (and only two) integer components amongst all $h_j$'s and $k_j$'s, to be different from zero. The corresponding sums, say $\Sigma_{h_i, h_j}$, $\Sigma_{k_i,k_j}$ with $1\le i<j\le d$ or $\Sigma_{h_i, k_j}$ with $1\le i,j\le d$, will be estimated similarly to \eqref{est Sigma h1h2}, by
\[
\Sigma_{h_i, h_j}\le \left(\frac{4d}{1+\varepsilon}\right)^2\alpha_i\alpha_j\,,\quad \Sigma_{k_i, k_j}\le\left(\frac{4d}{1+\varepsilon}\right)^2\beta_i\beta_j\,,\quad \Sigma_{h_i, k_j}\le\left(\frac{4d}{1+\varepsilon}\right)^2\alpha_i\beta_j\,.
\]
Eventually, summing over all different $\Sigma$'s as above we end up with
\begin{equation*}\label{Z2 est}
	\begin{split}
		\sum\limits_{(h,k)\in\mathbb Z^{2d}_{(2)}}&a_{(h,k)}\le\left(\frac{4d}{1+\varepsilon}\right)^2\left(\sum\limits_{1\le j_1,j_2\le d\,,\,\,j_1<j_2}\!\!\!\alpha_{j_1}\alpha_{j_2}\right.\\
		&\left.+\sum\limits_{1\le j_1,j_2\le d}\alpha_{j_1}\beta_{j_2}+\sum\limits_{1\le j_1,j_2\le d\,,\,\,j_1<j_2}\!\!\!\beta_{j_1}\beta_{j_2}\right)\,.
	\end{split}
\end{equation*}
In a similar way one can go on to estimate each of the sums $\sum\limits_{(h,k)\in\mathbb Z^{2d}_{(\ell)}}a_{(h,k)}$ as the integer $\ell$ increases from $1$ to $2d$. At the general step the resulting estimate will be
\begin{equation}\label{Zl est}
	\begin{split}
		&\sum\limits_{(h,k)\in\mathbb Z^{2d}_{(\ell)}}a_{(h,k)}\\
		&\le\left(\frac{4d}{1+\varepsilon}\right)^\ell\left(\sum\limits_{1\le j_1<\dots<j_\ell\le d}\!\!\!\alpha_{j_1}....\alpha_{j_\ell}+\dots+\sum\limits_{1\le j_1<\dots<j_\ell\le d}\!\!\!\beta_{j_1}....\beta_{j_\ell}\right)\,,
	\end{split}
\end{equation}
where the intermediate dots in the right-hand side above are used to mean any other sum of \lq\lq mixed\rq\rq products of $\ell$ numbers amongst both $\alpha_j$'s and $\beta_i$'s, occurring in all possible combinations. Let us notice that when the integer $\ell\le 2d$ is larger than $d$ then only products of both $\alpha_j$ and $\beta_i$ factors may occur in the right-hand side above.
\newline
In particular, as regards to the set $\mathbb Z^{2d}_{(2d)}$ of multi-integers $(h,k)$ with all their $2d$ scalar components different from zero, we will get
\begin{equation*}\label{Z2d est}
	\sum\limits_{(h,k)\in\mathbb Z^{2d}_{(2d)}}a_{(h,k)}\le\left(\frac{4d}{1+\varepsilon}\right)^{2d}\prod\limits_{j=1}^d\alpha_j\beta_j\,.
\end{equation*}
According to \eqref{Z split}, we sum up the estimates \eqref{Zl est} collected above to get
\begin{equation}\label{Sum est_0}
	\begin{split}
		&\sum\limits_{(h,k)\in\mathbb Z^{2d}_0}\prod\limits_{j=1}^d\left(1+\left\vert \frac{h_j}{\alpha_j}\right\vert\right)^{-1-\frac{1+\varepsilon}{2d}}\left(1+\left\vert\frac{k_j}{\beta_j}\right\vert\right)^{-1-\frac{1+\varepsilon}{2d}}\\
		&\le\sum\limits_{\ell=1}^{2d}\left(\frac{4d}{1+\varepsilon}\right)^\ell\left(\sum\limits_{1\le j_1<\dots<j_\ell\le d}\!\!\!\alpha_{j_1}....\alpha_{j_\ell}+\dots+\sum\limits_{1\le j_1<\dots<j_\ell\le d}\!\!\!\beta_{j_1}....\beta_{j_\ell}\right)\,.
	\end{split}
\end{equation}
Let us set $\theta:=\max\limits_{1\le j\le d}\left\{\alpha_j,\beta_j\right\}$. Then the products of $\ell$ factors $\alpha_{j_1}\dots\alpha_{j_\ell}$, $\beta_{j_1}\dots\beta_{j_\ell}$ and the mixed products involving both $\alpha_j$'s and $\beta_j$'s, in all possible combination (hidden behind the dots in the right-hand side above) are always majorized by $\theta^\ell$. The number of the aforementioned products of $\ell$ factors is equal to the number of {\em combinations of $2d$ elements with order $\ell$}, that is $\binom{2d}{\ell}$. Then 
\begin{equation}\label{Sum est_1}
	\sum\limits_{(h,k)\in\mathbb Z^{2d}_0}\prod\limits_{j=1}^d\left(1+\left\vert \frac{h_j}{\alpha_j}\right\vert\right)^{-1-\frac{1+\varepsilon}{2d}}\left(1+\left\vert\frac{k_j}{\beta_j}\right\vert\right)^{-1-\frac{1+\varepsilon}{2d}}\le\sum\limits_{\ell=1}^{2d}\binom{2d}{\ell}\left(\frac{4d}{1+\varepsilon}\theta\right)^\ell
\end{equation}   
follows from \eqref{Sum est_0}. From Newton's binomial formula we compute
\begin{equation}\label{New form}
	\begin{split}
		\sum\limits_{\ell=1}^{2d}&\binom{2d}{\ell}\left(\frac{4d}{1+\varepsilon}\theta\right)^\ell\\
		&=\sum\limits_{\ell=0}^{2d}\binom{2d}{\ell}\left(\frac{4d}{1+\varepsilon}\theta\right)^\ell-1=\left(1+\frac{4d}{1+\varepsilon}\theta\right)^{2d}-1
	\end{split}
\end{equation}   
Gathering \eqref{Vg sym est 0}, \eqref{Sum est_1} and \eqref{New form}, we obtain
\begin{equation}\label{Vg main est 2}
	\sum\limits_{(h,k)\in\mathbb Z^{2d}_0}\left\vert V_g\gamma\left(-\frac{k}{\beta},\frac{h}{\alpha}\right)\right\vert\le C_d\left[\left(1+\frac{4d}{1+\varepsilon}\theta\right)^{2d}-1\right]\mathcal K^\varepsilon_{\gamma, \hat\gamma}\mathcal K^\varepsilon_{g, \hat g}\,.
\end{equation}
Estimate \eqref{Vg main est 2} provides a bound of \eqref{Vggamma series} by a quite explicit function of the \lq\lq maximal mesh\rq\rq $\theta$ of the lattice  associated to the Gabor operator $S^{\alpha, \beta}_{g,\gamma}$. 

\medskip
\noindent
In view of Proposition \ref{InvGf_prop:1}, the invertibility of the operator $S^{\alpha, \beta}_{g,\gamma}$ in $\mathcal L(L^d)$ is guaranteed, provided that the right-hand side of estimate \eqref{Vg main est 2} is assumed to be smaller than $\left\vert(\gamma,g)_{L^2}\right\vert$, see \eqref{InvGf_eqt:2} and recall that the constant $C$ involved there is the one coming from tfs invariance \eqref{PST8}, which is $1$ in the case when $X=L^2(\mathbb R^d)$. This amounts to
\[
C_d\left[\left(1+\frac{4d}{1+\varepsilon}\theta\right)^{2d}-1\right]\mathcal K^\varepsilon_{\gamma, \hat\gamma}\mathcal K^\varepsilon_{g, \hat g}<\vert(\gamma, g)_{L^2}\vert\,.
\]
If we solve the above inequality with respect to $\theta$, we  get the smallness condition
\begin{equation}\label{stima_theta_sym}
	\theta<\frac{1+\varepsilon}{4d}\left[\left(\frac{\vert(\gamma,g)_{L^2}\vert}{C_d\mathcal K^\varepsilon_{\gamma,\hat\gamma}\mathcal K^\varepsilon_{g,\hat g}}+1\right)^{\frac1{2d}}-1\right]\,.
\end{equation}

\medskip
\noindent
Under the assumptions above, we have so proved the following
\begin{proposition}\label{prop_inv:2}
	For any $\varepsilon>0$, let $\gamma, g\in\mathcal S^\prime(\mathbb R^d)$ satisfy \eqref{g-gamma}, \eqref{hat g-gamma} and $(\gamma, g)_{L^2}\neq 0$. If the positive numbers $\alpha_j$, $\beta_j$ for $j=1,\dots, d$ are chosen such that
	\begin{equation}\label{stima theta fin_improved}
		\theta:=\max\limits_{1\le j\le d}\left\{\alpha_j,\beta_j\right\}<\frac{1+\varepsilon}{4d}\left[\left(\frac{\vert(\gamma,g)_{L^2}\vert}{C_d\mathcal K^\varepsilon_{\gamma, \hat\gamma}\mathcal K^\varepsilon_{g, \hat g}}+1\right)^{\frac1{2d}}-1\right]\,,
	\end{equation}
	with the constant $C_d$ defined in \eqref{cst} and $\mathcal K_{\gamma, \hat\gamma}$, $\mathcal K_{g,\hat g}$ defined by \eqref{L-inf-eps_sym eqt} and \eqref{Linf-eps_eqt:3}, then the generalized Gabor operator $S^{\alpha, \beta}_{g,\gamma}$ related to $L=diag(\alpha_1,\dots,\alpha_d,\beta_1,\dots,\beta_d)$ is invertible as an operator in $\mathcal L(L^2)$. Moreover, one has
	\begin{equation}\label{est inv_improved}
		\Vert \left(S^{\alpha, \beta}_{g,\gamma}\right)^{-1}\Vert_{\mathcal L(L^2)}\le\frac{\prod\limits_{j=1}^d\alpha_j\beta_j}{2\vert(\gamma, g)_{L^2}\vert-C_d\left[\left(1+\frac{4d\theta}{1+\varepsilon}\right)^{2d}-1\right]\mathcal K^\varepsilon_{\gamma, \hat\gamma}\mathcal K^\varepsilon_{g, \hat g}} 
	\end{equation}
\end{proposition} 
\begin{corollary}\label{GFB}
Consider $g\in \mathcal S'(\mathbb R^d)$ such that for some $\varepsilon >0$, 
$x^\beta\partial^\alpha g, x^\beta\partial_{x_j}\partial^\alpha g$, $x_jx^\beta\partial^\alpha g\in L^\infty_{\varepsilon}(\mathbb R^d)$ and $\omega^\beta\partial^\alpha \hat g, \omega^\beta\partial_{\omega_j}\partial^\alpha \hat g, \omega_j\omega^\beta\partial^\alpha \hat g\in L^\infty_{\varepsilon}(\mathbb R^d)$, for any $\alpha,\beta\in T^d$ and every integer $1\le j\le d$. Assume that the positive numbers $\alpha_j$, $\beta_j$ for $j=1,\dots, d$ satisfy
\begin{equation}\label{stima theta fin_improved_1}
		\theta:=\max\limits_{1\le j\le d}\left\{\alpha_j,\beta_j\right\}<\frac{1+\varepsilon}{4d}\left[\left(\frac{\Vert g\Vert^2_{L^2}}{C_d \mathcal K^\varepsilon_{g}\mathcal K^\varepsilon_{\hat g}}+1\right)^{\frac1{2d}}-1\right]\,,
	\end{equation}
where $\mathcal K^\varepsilon_{g}$ and $\mathcal K^\varepsilon_{\hat g}$ are the constants defined in \eqref{Linf-eps_eqt:3}.
\newline
Then the Gabor system $\mathcal G(g,\alpha, \beta):=\{M_{\beta k}T_{\alpha h} g\}_{h,k\in \mathbb Z^d}$ is a frame. Moreover we have as possible frame bounds
\begin{equation}\label{bounds}
\begin {array}{l}
A=\frac{2\Vert g\Vert^2_{L^2}-C_d\left[\left(1+\frac{4d\theta}{1+\varepsilon}\right)^{2d}-1\right]\mathcal K^\varepsilon_{g}\mathcal K^\varepsilon_{\hat g}} {\prod\limits_{j=1}^d\alpha_j\beta_j}\\
\\
B=\frac{2}{\prod\limits_{j=1}^d\alpha_j\beta_j} \Vert g\Vert^2_{L^2}.
\end{array}
\end{equation}

\end{corollary}


\appendix
\section{Proof of Proposition \ref{WS-prop:1}}\label{APP}
We need some preparatory lemmata.
\smallskip

For any $x\in\mathbb R$ we denote by $[x]$ the greatest integer number $\le x$, so that the inequalities $[x]\le x<[x]+1$ trivially follow.  
\begin{lemma}\label{WS-lemma:1}
For $a, b\in\mathbb R$ with $a\le b$, let $\#_{a,b}$ be the number of integers belonging to the closed interval $[a,b]$. Then $\#_{a,b}\ge 1$ if and only if $a\le[b]$. Under the previous condition, we also have
\begin{equation}\label{lem:1-eq:1}
\#_{a,b}\le [b-a]+1\,.
\end{equation}
\end{lemma}
\begin{proof}
As for the first assertion, it is clear from the definition of $[b]$ that $a\le [b]$ means that the real interval $[a,b]$ contains at least the integer $[b]$; conversely, if $k\in\mathbb Z$ is an integer belonging to $[a,b]$, the inequalities $a\le k\le[b]$ trivially follow  from definition of $[b]$ once again.
\newline
Assuming now that $a\le[b]$, let us turn to the second assertion of the lemma. Inequality \eqref{lem:1-eq:1} becomes trivial when $a=b$ (in which case the real interval reduces to the singleton $\{b\}$ and $[b]\equiv b$ is the only integer contained therein). Now for $a<b$ let us suppose there are at least $[b-a]+2$ subsequent integers that belong to $[a,b]$. This would imply that $[a,b]$ should contain at least the $[b-a]+1$ unitary sub-intervals bounded by the subsequent pairs of those integer numbers, giving rise to the contradiction $b-a\ge [b-a]+1$. Therefore \eqref{lem:1-eq:1} must be true.
\end{proof}
From \eqref{WS-norm}, it appears rather clear that proving the estimate \eqref{prop:1-eq:1} should be strictly related to evaluate the maximum number of points of the lattice $\Lambda:=L\mathbb Z^n$ that belong to the same cube $Q_r$ with given $r\in\mathbb Z^n$. The next result deals with the above question. 
\newline
In the sequel, let $L\in GL(n)$ be any real invertible matrix, which $(i,j)$-th entry will be denoted by $a_{i,j}$ for every $i,j=1,\dots,n$. Let $\mathcal L$ denote the {\em cofactor matrix} of $L$, which $(i,j)$-th entry will be denoted by $a^{i,j}$ for every $i,j=1,\dots,n$. Then as it is well-known
\begin{equation*}
L^{-1}=\frac1{det L}\mathcal L^T.
\end{equation*}
In agreement to the above, te $(i,j)$-th entry of $L^{-1}$ is given by $\displaystyle\frac{a^{j,i}}{det L}$ for all $i,j=1,\dots, n$. 
\begin{lemma}\label{WS-lemma:2}
For a real matrix $L\in GL(n)$ and a given vector $r=(r_1,\dots,r_n)\in\mathbb Z^n$, let $\#_{L,r}$ denote the number of elements of the discrete set
\begin{equation}\label{lem:2-eq:1}
Z_{L,r}:=\{\kappa\in\mathbb Z^n\,;\,\,\,L\kappa\in Q_r\}\,.
\end{equation}
Then we have that
\begin{equation}\label{lem:2-eq:2}
\#_{L,r}\le\prod_{j=1}^n\left(\left[\sum\limits_ {i=1}^{n}\left\vert\frac{a^{i,j}}{det L}\right\vert\right]+1\right)
\end{equation}
\end{lemma} 
\begin{proof}
It is first convenient to notice that whatever is $\kappa\in\mathbb Z^n$, the condition 
\[
L\kappa\in Q_r
\]
is equivalent to 
\begin{equation}\label{lem:2-eq:3}
\kappa\in L^{-1}Q_r\,,
\end{equation}
in view of the invertibility of $L$, namely $Z_{L,r}$ is the set of all $\kappa\in\mathbb Z^n$ belonging to the set $\mathcal P_r:=L^{-1}Q_r$. It is easy to observe that the set $\mathcal P_r$ is parametrized by the following vector equation
\begin{equation}\label{lem:2-eq:4}
x=L^{-1}(r+t)\,\quad\mbox{for}\,\,\,t:=(t_1,\dots,t_n)\in Q=[0,1]^n\,,
\end{equation}
and geometrically represents the convex $n-$parallelogram, which faces are (generally) transversal to the coordinate hyperplanes of $\mathbb R^n$ and which vertices are the $2^n$ points attained from \eqref{lem:2-eq:4} as long as the vector parameter $t$ spans the subset $\{0,1\}^n$ of $Q$. It is then clear that $Z_{L,r}$ is included into the parallelogram, say $\mathcal R_r$, which faces are parallel to the coordinate hyperplanes, generated from the Cartesian product of the $n$ real intervals obtained from projecting all vertices of $\mathcal P_r$ orthogonally on each of the coordinate hyperplanes. To go on, first notice that the $j-$th coordinate of all vertices of $\mathcal P_r$ are the values computed from
\begin{equation}\label{lem:2-eq:5}
x_j^{(t)}=\frac1{det L}\left(a^{1,j}(r_1+t_1)+\dots+a^{n,j}(r_n+t_n)\right)\,,
\end{equation}   
for all possible $t=(t_1\,\dots, t_n)\in\{0,1\}^n$, for every $j=1,\dots,n$. If we set
\begin{equation}
x_j^{\min}:=\min\limits_{t\in\{0,1\}^n}\{x_j^{(t)}\}\quad\mbox{and}\quad x_j^{\max}:=\max\limits_{t\in\{0,1\}^n}\{x_j^{(t)}\}\,,
\end{equation}
then the parallelogram $\mathcal R_r$, which we are looking for, is defined as 
\begin{equation}\label{lem:2-eq:6}
\mathcal R_r:=[x_1^{\min}, x_1^{\max}]\times\dots\times[x_n^{\min}, x_n^{\max}]\,;
\end{equation}
of course $\mathcal P_r\subseteq\mathcal R_r$, so the number of $\kappa\in\mathbb Z^n$ contained in $\mathcal R_r$ will be certainly greater than or equal to the number $\#_{L,r}$ of members of $Z_{L,r}$.
\newline
Notice also that for any $j=1,\dots,n$ the coefficients $\displaystyle\frac{a^{1,j}}{det L},\dots,\displaystyle\frac{a^{n,j}}{det L}$ involved in \eqref{lem:2-eq:5} cannot be all zero, otherwise the matrix $L^{-1}$ should have some zero row contradicting the invertibility of $L^{-1}$ (that is the invertibility of $L$). 
\newline
Let $n_+$, $n_-$ and $n_0$ denote respectively the number of positive, negative and zero elements of the $j-$th row $\left\{\displaystyle\frac{a^{i,j}}{det L}\,,\,\,\,i=1,\dots,n\right\}$ (so that $n=n_++n_-+n_0$); assuming (without loss of generality) that $n_\pm$ and $n_0$ are all strictly positive and
\[
\begin{split}
& \frac{a^{i,j}}{det L}>0\,,\quad i=1,\dots,n_+\,,\\
& \frac{a^{i,j}}{det L}<0\,,\quad i=n_++1,\dots,n_++n_-\,,\\
& \frac{a^{i,j}}{det L}=0\,,\quad i=n_++n_-+1,\dots,n_++n_-+n_0=n\,,
\end{split}
\]
then \eqref{lem:2-eq:5} reduces to
\begin{equation}\label{lem:2-eq:7.0}
x_j^{(t)}=\sum\limits_ {i=1}^{n_+}\frac{a^{i,j}}{det L}(r_i+t_i)+\sum\limits_{i=n_++1}^{n_++n_-}\frac{a^{i,j}}{det L}(r_i+t_j)\,,
\end{equation} 
and it is easy to derive
\begin{equation}\label{lem:2-eq:7}
\begin{split}
& x_j^{\min}=\sum\limits_ {i=1}^{n_+}\frac{a^{i,j}}{det L}r_i+\sum\limits_{i=n_++1}^{n_++n_-}\frac{a^{i,j}}{det L}(r_i+1)\,,\\
& x_j^{\max}=\sum\limits_ {i=1}^{n_+}\frac{a^{i,j}}{det L}(r_i+1)+\sum\limits_{i=n_++1}^{n_++n_-}\frac{a^{i,j}}{det L}r_i\,.
\end{split}
\end{equation} 
In view of Lemma \ref{WS-lemma:1}, the number of integers belonging to the interval $[x_j^{\min}, x_j^{\max}]$, say $\#_j$, satisfies
\begin{equation}\label{lem:2-eq:8}
\#_j\le[x_j^{\max}-x_j^{\min}]+1
\end{equation} 
and for the lenght $x_j^{\max}-x_j^{\min}$ of the interval, we easily derive from \eqref{lem:2-eq:7}
\begin{equation}\label{lem:2-eq:9}
x_j^{\max}-x_j^{\min}=\sum\limits_ {i=1}^{n_+}\frac{a^{i,j}}{det L}-\sum\limits_{i=n_++1}^{n_++n_-}\frac{a^{i,j}}{det L}=\sum\limits_ {i=1}^{n}\left\vert\frac{a^{i,j}}{det L}\right\vert\,.
\end{equation}
Because of \eqref{lem:2-eq:6}, we  conclude that the number of points $\kappa\in\mathbb Z^n$ belonging to $\mathcal R_r$, say $\#_{\mathcal R_r}$, satisfies
\begin{equation}\label{lem:2-eq:10}
\#_{\mathcal R_r}\le\prod_{j=1}^n\left(\left[\sum\limits_ {i=1}^{n}\left\vert\frac{a^{i,j}}{det L}\right\vert\right]+1\right)
\end{equation}
from which \eqref{lem:2-eq:2} follows at once. 
\end{proof} 
We are now in the position to prove Proposition \ref{WS-prop:1}.

\medskip
\noindent
\begin{proof}[\textit {Proof of Proposition \ref{WS-prop:1}}]. Because the family of cubes $\{Q_r\}_{r\in\mathbb Z^n}$ is a covering of $\mathbb R^n$, for a given $f\in W(L^1_v)$ we can rearrange the terms of the series in the left-hand side of \eqref{prop:1-eq:1} so as to get 
\begin{equation}\label{prop:1-eq:2}
\sum\limits_{\kappa\in\mathbb Z^n}v(L\kappa)\vert f(L\kappa)\vert\le\sum\limits_{r\in\mathbb Z^n}\sum\limits_{\kappa\in Z_{L,r}}v(L\kappa)\vert f(L\kappa)\vert\,,
\end{equation}
where $Z_{L,r}$ was defined in \eqref{lem:2-eq:1}. Now for any given $r\in\mathbb Z^n$, in view of \eqref{WS-eq:1}
\[
v(L\kappa)\vert f(L\kappa)\vert\le M_v v(r)\Vert\chi_r f\Vert_{L^\infty}\,,
\]
where $M_v$ is the constant in \eqref{WS-eq:1}. Using the latter to majorize each term of the sum over $r$ in the right-hand side of \eqref{prop:1-eq:2} leads to
\[
\sum\limits_{r\in\mathbb Z^n}\sum\limits_{\kappa\in Z_{L,r}}v(L\kappa)\vert f(L\kappa)\vert\le M_v\sum\limits_{r\in\mathbb Z^n}v(r)\Vert\chi_r f\Vert_{L^\infty}\#_{L,r}\,,
\]
from which \eqref{prop:1-eq:1} follows at once, using Lemma \ref{WS-lemma:2} to estimate $\#_{L,r}$ in the right-hand side above, where the constant $C_{L,v}$ is given by
\eqref{prop:1-est:3}.
\end{proof}


\medskip

{\small
Gianluca Garello\\
Department of Mathematics \lq\lq G. Peano\rq\rq \\
University of Torino\\
Via Carlo Alberto 10, I-10123 Torino, Italy\\
gianluca.garello@unito.it}
\medskip

\noindent {\small
Alessandro Morando\\
INdAM Unit \& Department of Civil, Environmental, Architectural Engineering and Mathematics (DICATAM)\\
University of Brescia\\
Via Valotti 9, I-25133 Brescia, Italy\\
alessandro.morando@unibs.it}
\end{document}